\documentclass[11pt,reqno]{amsart}

\usepackage{amsmath, amsfonts, amsthm, amssymb}

\numberwithin{equation}{section}
\newtheorem{Theorem}{Theorem}[section]
\newtheorem{Lemma}{Lemma}[section]

\theoremstyle{definition}

\theoremstyle{remark}
\newtheorem{Remark}{Remark}[section]
\newtheorem{Proposition}{Proposition}[section]

\renewcommand{\r}{\rho}
\renewcommand{\t}{\theta}

\def\i{\varepsilon}
\renewcommand{\u}{{\bf u}}

\newcommand{\R}{{\mathbb R}}
\newcommand{\Dv}{{\rm div}}

\newcommand{\T}{{\mathcal T}}
\def\f{\frac}

\def\ov{\overline}

\def\D{\Delta }

\def\hf1{^\f{1}{1-\xi^2}}

\def\be{\begin{equation}}
\def\en{\end{equation}}
\def\bs{\begin{split}}
\def\es{\end{split}}

\author{Xianpeng HU and Dehua Wang}
\address{Courant Institute of Mathematical Sciences, New York University, New York, NY 10012.}
\email{xianpeng@cims.nyu.edu}
\address{Department of Mathematics, University of Pittsburgh,
                           Pittsburgh, PA 15260.}
\email{dwang@math.pitt.edu}

\title[Incompressible Magnetohydrodynamics Limit of VMB]
{Incompressible Magnetohydrodynamic limit of the
Vlasov-Maxwell-Boltzmann Equations}


\date{\today}

\begin{document}

\begin{abstract}
The hydrodynamic limit of the Vlasov-Maxwell-Boltzmann equations is considered for weak solutions.
Using relative entropy estimate about an absolute Maxwellian, an
incompressible Electron-Magnetohydrodynamics-Fourier limit for
solutions of the Vlasov-Maxwell-Blotzmann equations over any
periodic spatial domain in $\R^3$ is studied. It is shown that
any properly scaled sequence of renormalized solutions of the
Vlasov-Maxwell-Boltzmann equations has fluctuations that (in the
weak $L^1$ topology) converge to an infinitesimal Maxwellian with
fluid variables that satisfy the incompressibility and Boussinesq
relations. It is also shown that the limits of the velocity, the electric
field, and the magnetic field are governed by a weak solution of
an incompressible electron-magnetohydrodynamics system for all
time.
\end{abstract}

\maketitle

\section{Introduction}
The hydrodynamic models such as the Euler or Navier-Stokes
equations were first established by applying Newton's second law
of motion to infinitesimal volume elements of the fluid under
consideration; while the kinetic equations are the mathematical models
used to describe the dilute particle gases at an intermediate
scale between microscopic and macroscopic level with applications
in a variety of sciences such as plasma, astrophysics, aerospace
engineering, nuclear engineering, particle-fluid interactions,
semiconductor technology, social sciences, and biology. If the
particles interact only through a repulsive conservative
interparticle force with finite range, then at low enough
densities this range will be much smaller than the interparticle
spacing. In that regime, 
the evolution of the density of particles $F=F(x,\xi,t)$ is governed by the classical
Vlasov-Maxwell-Boltzmann equaitons (VMB) \cite{RTG, YG, HW}:
\begin{subequations}\label{me1}
\begin{align}
&\f{\partial F}{\partial t}+\xi\cdot\nabla_x F+e(E+\xi\times
B)\cdot\nabla_\xi F=\mathcal{Q}(F, F),\quad x\in \R^3,\quad \xi\in
\R^3,\quad
t\ge 0,\\
&\f{1}{c^2}\f{\partial E}{\partial t}-\nabla\times B=-\mu_0
j,\quad \Dv B=0,\quad
\textrm{on}\quad \R^3_x\times(0,\infty),\\
&\f{\partial B}{\partial t}+\nabla\times E=0,\quad\Dv
E=\f{\r}{\eta_0},\quad
\textrm{on}\quad \R^3_x\times(0,\infty),\\
&\r=e\int_{\R^3}Fd\xi, \quad j=e\int_{\R^3}F\xi d\xi, \quad \textrm{on}\quad \R^3_x\times(0,\infty),
\end{align}
\end{subequations}
where the nonnegative function $F(t,x,\xi)$ is the density of
particles with velocity $\xi$  at time $t$ and position $x$  under the effect of
the Lorentz force $$E+\xi\times B,$$  $E$ is the electric field,
and $B$ is the magnetic field. The function $j$ is called the
current density, while the function $\r$ is the charge density.
The constant $e$ is the charge of the electron. The constant
$c$ is the speed of light. The coefficients
$\mu_0$ and $\eta_0$ are the magnetic permeability and the
electric permittivity of the plasma in the vacuum (see \cite{BS,MG}),   satisfying
$\mu_0\eta_0 c^2=1.$
The collison operator $\mathcal{Q}(F,F)$ is defined as
$$\mathcal{Q}(F,F)=\int_{\R^3}d\xi^*\int_{S^2}d\omega\,
b(\xi-\xi_*,\omega)(F'F'_*-FF_*),$$
where the nonnegative function $b(\xi, \omega)$ given for $\xi\in\R^3$ and $\omega\in S^2$ (the unit sphere in $\R^3$)  is called the collision kernel,  and
$$F_*=F(t,x,\xi_*),\quad F'=F(t, x,\xi'), \quad F'_*=F(t, x,\xi'_*), $$  with
$$\xi'=\xi-(\xi-\xi_*,\omega)\omega,$$
$$\xi_*'=\xi_*+(\xi-\xi_*,\omega)\omega,$$   yielding
one convenient parametrization of the set of solutions to the law
of elastic collisions:
\begin{equation}\label{ei}
\begin{cases}
\xi'+\xi'_*=\xi+\xi_*,\\
|\xi'|^2+|\xi'_*|^2=|\xi|^2+|\xi_*|^2.
\end{cases}
\end{equation}
The interpretation of $\xi$, $\xi_*$, $\xi'$, $\xi_*'$ is the
following: $\xi, \xi_*$ are the velocities of two colliding
molecules immediately before collision while $\xi', \xi'_*$ are
the velocities immediately after the collision.
We will consider the initial value problem of system \eqref{me1} with the initial condition:
\begin{equation}\label{BI}
(F, E, B)|_{t=0}=(F^0(x,\xi), E^0(x),
B^0(x))\quad\textrm{for}\quad x\in\R^3,\quad \xi\in\R^3.
\end{equation}

On the macroscopic level, the incompressible
Electron-Magnetohydrodynamics-Fourier equations describe the
evolution of the velocity field $\u=\u(t,x)$ of an idealized fluid
over a given spatial domain in $\R^3$ under the magnetic field
$B=B(t,x)$ and the electronic field $E=E(t,x)$, and take the form (cf. \cite{BMP})
\begin{subequations}\label{IM}
\begin{align}
&\partial_t\u+\u\cdot\nabla\u-\mu\D \u+\nabla
p-\alpha eE=(\nabla\times B)\times B,\label{IM1}\\
&\partial_t B+\nabla\times E=0, \quad j=\nabla\times B=e\u,\label{IM2}\\
&\partial_t\t+\u\cdot\t=\kappa \D\t,\quad \nabla_x(h+\t)=0,\label{IM4}\\
&\Dv \u=0, \qquad \Dv B=0,\label{IM6}
\end{align}
\end{subequations}
with
$$\alpha=\f{1}{3(2\pi)^{\f{3}{2}}}\int_{\R^3}|\xi|^2\exp\left(-\f{|\xi|^2}{2}\right)d\xi,$$
where $p, \t, E, h$ denote the pressure, temperature, electric field, and density respectively.
The initial value problem will also be considered for system \eqref{IM} with the initial data:
\begin{equation}\label{IMIN}
(\u, B, \t)|_{t=0}=(\u_0(x), B_0(x), \t_0(x)), \quad x\in\R^3,
\end{equation}
where
$$\u_0, B_0\in\{v\in L^2(\R^3):\Dv v=0\quad\textrm{in}\quad
\mathcal{D}'\}\quad\text{and}\quad \t_0\in L^2(\R^3).$$
We call $(\u, p, B, E, \t)$ a $\textit{weak
solution}$ to \eqref{IM}-\eqref{IMIN} if $(\u, p, B, E)$ is a Leray's
solution of the incompressible electron-magnetohydrodynamic
equation \eqref{IM1}-\eqref{IM2} under the constraints \eqref{IM6}
with initial data in \eqref{IMIN}, while $\t$ is a weak solution in
the sense of distributions to \eqref{IM4} with the initial data in \eqref{IMIN}.

The motivation of this paper is to find a scaling and
 verify mathematically the transition from the microscopic model
\eqref{me1} to the macroscopic model \eqref{IM} as some parameter
vanishes. One of the main objectives is to connect the
DiPerna-Lions theory of global renormalized solutions of the
Boltzmann equation with the Leray theory of global weak solutions of
the incompressible fluid equations in a periodic spatial domain
$\mathcal{T}=[0,1]^3\subset\R^3$.
More precisely,  we consider the hydrodynamic limit of the Vlasov-Maxwell-Boltzmann equations  for weak solutions in this paper.
Using relative entropy estimate about an absolute Maxwellian, an
incompressible Electron-Magnetohydrodynamics-Fourier limit for
solutions of the Vlasov-Maxwell-Blotzmann equations over
periodic spatial domains in $\R^3$ is studied. It is shown that
any properly scaled sequence of renormalized solutions of the
Vlasov-Maxwell-Boltzmann equations has fluctuations that (in the
weak $L^1$ topology) converge to an infinitesimal Maxwellian with
fluid variables that satisfy the incompressibility and Boussinesq
relations. It is also shown that the limits of the velocity, the electric
field, and the magnetic field are governed by a weak solution of
an incompressible electron-magnetohydrodynamics system for all
time.

We note that if the local conservation laws of
momentum and energy are assumed, the similar result for the
Navier-Stokes limit from the Boltzmann equation was verified in
\cite{BGL1, BGL}.  Later, without making any nonlinear weak
compactness hypothesis,  Golse and Saint-Raymond in their
breakthrough paper \cite{GL} established the incompressible
Navier-Stokes limit of Boltzmann equations with bounded kernels,
where they used the entropy dissipation rate to decompose the
collision operator in a new way and used a new $L^1$ averaging
theory to prove the compactness assumption. Recently, Levermore
and Masmoudi \cite{LM} extended those results to a much wider
class of collision kernels.
As the collision is neglected, the Vlasov-Maxwell-Boltzmann
equations become the Vlasov-Maxwell equations.  For
Vlasov-Maxwell equations, the mathematical verification from the
weak solution of the Vlasov-Maxwell equations to the incompressible
models has been conducted in \cite{BMP, BG, GL3, PS}. When the
solution is smooth, authors in \cite{BELM, JJ} considered the Hilbert expansion of solutions to the Vlasov-Maxwell-Boltzmann equations.
For other related results on the Boltzmann equation, see  \cite{FFM, GLD, GL2, MS, SL} and the references therein. Also for the large-time behavior, stability, and regularity of solutions to the Vlasov-Maxwell-Boltzmann equations, see \cite{Duan, DS1, GS1, HW} and the references therein.

This paper is organized as follows. In Section 2, we state the
formal scalings, the relative entropy, the technical assumptions,
and the main result. Section 3 is devoted to a list of \textit{a
priori} estimates on the fluctuations of the density from the
relative entropy. In Section 4, we consider the limit of the
Maxwell equations. Section 5 will concentrate on the vanishing of
conservation defects. And finally, in Section 6 we will give the detailed proof
of our main Theorem \ref{T1}.

\bigskip\bigskip

\section{Dimensionless Analysis, Preliminary, and Main Results}

In this section, we first introduce the scaling of \eqref{me1},
then deduce the relative entropy, and finally state the main
result.

To begin with, we now focus on the nondimensional form of the
Vlasov-Maxwell-Boltzmann equations. This form is motivated by the
fact that the incompressible Electron-Magnetohydrodynamics-Fourier
system \eqref{IM} can be formally derived from the
Vlasov-Maxwell-Boltzmann equations through a scaling, when the
density $F$ is close to a spatially homogeneous Maxwellian
$M=M(\xi)$ that has the same total mass, momentum, and energy as
the initial data. To this end, we introduce
$$t=t_*\hat{t},\qquad x=x_*\hat{x},\qquad \xi=\xi_*\hat{\xi},$$
$$F=\f{1}{\mu_0 \xi_*^3 x_*^2}\hat{F},\qquad E=\f{\xi_*}{t_*}\hat{E},\qquad B=\f{1}{t_*}\hat{B},$$
and
$$b=\f{x_*}{\eta_0\xi_*}\hat{b},$$
where the constants $t_*, x_*,\xi_*$ are the characteristic time,  characteristic distance, and characteristic  speed; see \cite{BS} for more physical interpretations of these constants. The non-relativistic effect
requires
$$\xi_*=\f{x_*}{t_*}, \quad\textrm{and}\quad
\i=\left(\f{\xi_*}{c}\right)^2\ll 1.$$ Substituting those new
variables back to \eqref{me1}, and dropping hats, we obtain

\begin{subequations}\label{me3111}
\begin{align}
&\f{\partial F}{\partial t}+\xi\cdot\nabla_x F+e(E+\xi\times
B)\cdot\nabla_\xi F=\f{1}{\i}\mathcal{Q}(F, F),\label{me31111}\\
&\i\f{\partial E}{\partial t}-\nabla\times B=-j,\label{me31112}\\
&\f{\partial B}{\partial t}+\nabla\times E=0,\label{me31113}\\
&\Dv B=0,\quad\quad\Dv E=\f{\r}{\i},\label{me31114}
\end{align}
\end{subequations}
where the coefficient $\i$ 
is usually refereed as the dimensionless mean free
path or Knudsen number.

Since the incompressible flow is the large-scale low-frequency
fluid-like behavior of a plasma system (\cite{BS, MG}), we need to
further scale the time to the order of $\i^{-1}$. For this
purpose, in the system \eqref{me3111}, we further introduce the
scaling as
\begin{equation*}
\begin{split}
&\tilde{t}=\i t,\qquad \tilde{x}=\i x,\qquad \tilde{\xi}=\i \xi,\\
&\tilde{F}=\f{1}{\i^5} F,\qquad\tilde{E}=\f{1}{\i}E,\qquad \tilde{B}=\f{1}{\i}B,\\
&\textrm{and}\\
&\tilde{b}=\i^2 b.
\end{split}
\end{equation*}
Then substituting the above scaling  back into \eqref{me3111},
and dropping tildes, we obtain

\begin{subequations}\label{me3}
\begin{align}
&\i\f{\partial F}{\partial t}+\xi\cdot\nabla_x F+e\i(\i
E+\xi\times
B)\cdot\nabla_\xi F=\f{1}{\i}\mathcal{Q}(F, F),\label{me31b}\\
&\i\f{\partial E}{\partial t}-\nabla\times B=-\f{j}{\i},\label{me32}\\
&\f{\partial B}{\partial t}+\nabla\times E=0,\label{me33}\\
&\Dv B=0,\quad\quad\Dv E=\f{\r}{\i},\label{me34}
\end{align}
\end{subequations}

The incompressible Electron-Magnetohydrodynamics-Fourier equations
will be obtained when F is close to the absolute Maxwellian
$M$\footnote{The absolute Maxwellian is given as
\begin{equation}\label{mm}
M(\xi)=\f{1}{(2\pi)^{\f{3}{2}}}\exp\left(-\f{1}{2}|\xi|^2\right)
\end{equation} and corresponds to the spatially homogeneous fluid state
with its density and temperature equal to $1$, bulk velocity equal
to $0$ and no effect from the electric field and the magnetic
field.} with order $\i$. Motivated by \cite{BGL, GL, LM}, we set
$F=MG$. Recasting the system \eqref{me3} for $G$ yields
\begin{subequations}\label{me31}
\begin{align}
&\i\f{\partial G}{\partial t}+\xi\cdot\nabla_x G+e\i(\i
E+\xi\times
B)\cdot\nabla_\xi G-e\i^2 E\cdot \xi G=\f{1}{\i}Q(G, G),\\
&\i\f{\partial E}{\partial t}-\nabla\times B=-\f{j}{\i},\\
&\f{\partial B}{\partial t}+\nabla\times E=0,\label{me313}\\
&\Dv B=0,\quad \quad \Dv E=\f{\r}{\i}.
\end{align}
\end{subequations}
where the collision operator is now given by
$$Q(G, G)=\int_{\R^3}\int_{S^2}(G'_*G'-G_*G)b(\xi_*-\xi,
\omega)d\omega M_*d\xi_*,$$ where $M_*=M(\xi_*)$.

\subsection{Relative Entropy}

For any pair of measurable functions $f$ and $g$ defined a.e. on
$\R^3\times \R^3$ and satisfying $f\ge 0$ and $g>0$ a.e., we use
the following notation for the relative entropy
\begin{equation}\label{re1}
H(f|g)=\int_{\T}\int_{\R^3}\left[f\ln\left(\f{f}{g}\right)-f+g\right]d\xi
dx\in [0,\infty],
\end{equation}
which is a way to measure how far $f$ is away from $g$. We are
interested in the evolution of
\begin{equation}\label{re2}
\mathcal{H}_\i(t)=\i H(F_\i|M)+\f{\i^3}{2}\int_{\T}(\i
|E_\i|^2+|B_\i|^2)dx,
\end{equation}
where $(F_\i, E_\i, B_\i)_{\{\i>0\}}$ are renormalized solutions
(see definition in Section 2.4) of Vlasov-Maxwell-Boltzmann
equations \eqref{me3}. This quantity contains the information from
the standard (rescaled) $L^2$ norm of the electro-magnetic field
and from the relative entropy between the renormalized solution
$F_\i(t,x,\xi)$ and the absolute Maxwellian $M$.

The following lemma is devoted to the study of the evolution of
the relative entropy, deduced from
\begin{equation}\label{re3}
\begin{split}
\f{d}{dt}\mathcal{H}_\i&=\i\int_{\T}\int_{\R^3}\partial_t F_\i(\ln
F_\i-\ln M)d\xi dx\\&\quad+\f{\i^3}{2}\f{d}{dt}\int_{\T}(\i
|E_\i|^2+|B_\i|^2)dx.
\end{split}
\end{equation}

\begin{Lemma}\label{re}
Let $(F_\i, E_\i, B_\i)$ be a renormalized solution (refer to the
definition in Section 2.2.1 below) to \eqref{me3}. Then
$\mathcal{H}_\i(t)$ satisfies the differential inequality:
\begin{equation}\label{re6}
\f{d}{dt}\mathcal{H}_\i(t)+\f{1}{4\i}\int_{\T}\int_{\R^3}\ln\left(\f{{F_\i}'_*{F_\i}'}{{F_\i}_*{F_\i}}\right)({F_\i}'_*{F_\i}'-{F_\i}_*{F_\i})d\xi
dx\le 0.
\end{equation}
\end{Lemma}

\begin{proof}
In view of \cite{HW}, the inequality will follow from the lower
semi-continuity of the weak convergence and an equality version
when solutions are smooth. Thus, we will assume that those
solutions are smooth. Observing that
$$\partial_t F_\i\ln F_\i=\partial_t(F_\i\ln F_\i)-\partial_t
F_\i,$$ we obtain
\begin{equation*}
\begin{split}
\i\int_{\T}\int_{\R^3}\partial_t F_\i\ln F_\i d\xi
dx=-\f{1}{4\i}\int_{\T}\int_{\R^3}\ln\left(\f{{F_\i}'_*{F_\i}'}{{F_\i}_*{F_\i}}\right)({F_\i}'_*{F_\i}'-{F_\i}_*{F_\i})d\xi
dx,
\end{split}
\end{equation*}
and, by \eqref{me3}
\begin{equation*}
\begin{split}
\i\int_{\T}\int_{\R^3}\partial_t F_\i\ln M d\xi
dx=-\i^2e\int_{\T}\int_{\R^3}F_\i E_\i\cdot\xi d\xi dx.
\end{split}
\end{equation*}
Here, we used the following identity twice (see \cite{GL})
$$\int_{\T}Q(f,f)\zeta(\xi)d\xi=\f{1}{4}\int_{\T}\int_{\R^3}d\xi d\xi_*\int_{S^2}d\omega
B(f'f'_*-ff_*)[\zeta+\zeta_*-\zeta'-\zeta'_*].$$
 Hence,
\begin{equation}\label{re4}
\begin{split}
\i\int_{\T}\int_{\R^3}\partial_t F_\i(\ln F_\i-\ln M) d\xi
dx&=-\f{1}{4\i}\int_{\T}\int_{\R^3}\ln\left(\f{{F_\i}'_*{F_\i}'}{{F_\i}_*{F_\i}}\right)({F_\i}'_*{F_\i}'-{F_\i}_*{F_\i})d\xi
dx\\&\quad+e\i^2\int_{\T}\int_{\R^3}F_\i E_\i\cdot\xi d\xi dx.
\end{split}
\end{equation}

On the other hand, multiplying equation \eqref{me32} by $E_\i$,
equation \eqref{me33} by $B_\i$, integrating them in $x$ over
$\R^3$ and then summing them together, we obtain,
\begin{equation}\label{re5}
\f{d}{dt}\int_{\T}(\i|E_\i|^2+|B_\i|^2)dx=-\f{2}{\i}\int_{\T}E_\i\cdot
j_\i dx=-e\f{2}{\i}\int_{\T}\int_{\R^3}E_\i\cdot \xi F_\i d\xi dx.
\end{equation}

Substituting \eqref{re5} back into \eqref{re4}, we obtain
\begin{equation*}
\begin{split}
\i\int_{\T}\int_{\R^3}\partial_t F_\i(\ln F_\i-\ln M) d\xi
dx&=-\f{1}{4\i}\int_{\T}\int_{\R^3}\ln\left(\f{{F_\i}'_*{F_\i}'}{{F_\i}_*{F_\i}}\right)({F_\i}'_*{F_\i}'-{F_\i}_*{F_\i})d\xi
dx\\&\quad-\f{\i^3}{2}\f{d}{dt}\int_{\T}(\i|E_\i|^2+|B_\i|^2)dx,
\end{split}
\end{equation*}
which is exactly an equality version of \eqref{re6}.
\end{proof}

{\bf Notations.} In order to avoid unnecessary constants in the
sequel, we will assume that the nondimensionalization has the
following normalizations:
$$\int_{S^2} d\omega=1,\qquad \int_{\R^3} Md\xi=1,$$
associated with the domain $S^2$, and $\R^3$ respectively;
\begin{equation*}
\int_{\T}\int_{\R^3}G^0M d\xi dx=1,\qquad \int_{\T}\int_{\R^3}\xi
G^0 M d\xi dx=0,
\end{equation*}
$$\int_{\T}\int_{\R^3}\f{1}{2}|\xi|^2G^0M d\xi dx=\f{3}{2},$$
associated with the initial data; and
$$\int_{\R^3}\int_{\R^3}\int_{S^2}b(\xi_*-\xi, \omega)d\omega
M_*d\xi_* M d\xi=1,$$ associated with the Boltzmann kernel.

Since $Md\xi$ is a positive unit measure on $\R^3$, we denote by
$<\eta>$ the average over this measure of any integrable function
$\eta=\eta(\xi)$,
$$<\eta>=\int_{\R^3}\eta Md\xi.$$
Since $$d\mathcal{M}=b(\xi_*-\xi, \omega)d\omega M_*d\xi_* Md\xi$$
is a non-negative unit measure on $\R^3\times\R^3\times S^2$, we
denote by $\ll\tau\gg$ the average over this measure of any
integrable function $\tau=\tau(\xi, \xi_*, \omega)$,
$$\ll\tau\gg=\int_{\R^3}\tau d\mathcal{M}.$$ The collision measure
$d\mathcal{M}$ is invariant under the transformations
$$(\omega, \xi_*, \xi)\rightarrow (\omega, \xi, \xi_*),\quad (\omega, \xi_*, \xi)\rightarrow (\omega, \xi',
\xi'_*),$$ which are called collisional symmetries (cf. \cite{BGL,
GL}).

Now, we can explain Lemma \ref{re} in terms of $G_\i$ as follows:
\begin{equation}\label{reg}
\begin{split}
&\i\f{d}{dt}\int_{\T}\left< G_\i\ln G_\i-G_\i+1\right>
dx+\f{\i^3}{2}\f{d}{dt}\int_{\T}(\i
|E_\i|^2+|B_\i|^2)dx\\&\quad+\f{1}{4\i}\int_{\T}\left\langle\!\!\!\left\langle
\ln\left(\f{{G_\i}'_*{G_\i}'}
{{G_\i}_*{G_\i}}\right)({G_\i}'_*{G_\i}'-{G_\i}_*{G_\i})\right\rangle\!\!\!\right\rangle
dx\le 0.
\end{split}
\end{equation}
If $G_\i$ solves the VMB equations \eqref{me31}, then inequality
\eqref{reg} implies
\begin{equation}\label{m13}
\mathcal{H}_\i(t)+\f{1}{\i}\int_0^t\mathcal{R}(G_\i(s))ds=\mathcal{H}_\i(0),
\end{equation}
where $\mathcal{H}_\i(t)$ is the entropy functional
\begin{equation}\label{m14}
\mathcal{H}_\i(t)=\i\int_{\R^3}<G_\i\ln G_\i-G_\i+1>dx+\f{\i^3}{2}
\int_{\R^3}(\i|E_\i|^2+|B_\i|^2)dx,
\end{equation}
and $R(G)$ is the entropy dissipation rate functional
\begin{equation}\label{m15}
\mathcal{R}(G)=\int_{\R^3}\f{1}{4}\left\langle\!\!\!\left\langle
\ln\left(\f{{G_\i}'_*{G_\i}'}
{{G_\i}_*{G_\i}}\right)({G_\i}'_*{G_\i}'-{G_\i}_*{G_\i})\right\rangle\!\!\!\right\rangle
dx.
\end{equation}

This choice of $\mathcal{H}_\i$ as the entropy functional
\eqref{m14} is based on the fact that its integrand is a
non-negative strictly convex function of $G$ with a minimum value
of zero at $G=1$. Indeed for any $G$,
\begin{equation}\label{m16}
H(G)\ge 0,\qquad\textrm{and}\qquad H(G)=0\qquad\textrm{if and only
if}\quad G=1.
\end{equation}
Here $H(G)$ is called the relative entropy with respect to the
absolute equilibrium $G=1$ which provides a natural measure of the
proximity of $G$ to that equilibrium.

We can expect that, the terms involving the entropy
$\mathcal{H}_\i$ measure the proximity of $G_\i$ and $G_\i^0$ to
the absolute equilibrium value of $1$. On the other hand, the
terms involving the dissipation rate $\mathcal{R}$, can be
understood to measure the proximity of $G_\i$ to any Maxwellian
through their characterization.

\subsection{Global Solutions}
In order to mathematically justify the incompressible
Electron-Magnetohydrodynamics-Fourier limit of the
Vlasov-Maxwell-Boltzmann equations, we must make precise:
\begin{itemize}
\item the notion of solutions for the Vlasov-Maxwell-Boltzmann
equations;
\item the notion of solutions for the incompressible
Electron-Magnetohydrodynamics-Fourier system \eqref{IM}.
\end{itemize}
Ideally, these solutions should be global while the bounds are
physically natural. We therefore work in the setting of
DiPerna-Lions renormalized solutions for the
Vlasov-Maxwell-Boltzmann equations, and in the setting of Leray
solutions for the incompressible
Electron-Magnetohydrodynamics-Fourier system. These theories have
the virtues of considering physically natural classes of initial
data.

\subsubsection{Renormalized solutions to the Vlasov-Maxwell-Boltzmann equations}
In the spirit of the DiPerna-Lions theory for the Boltzmann
equation and the idea in Hu-Wang \cite{HW}, modified slightly for the
periodic box, it is possible to show the weak stability of global
weak solutions to a whole class of formally equivalent
initial-value problems. More precisely, let $G_\i\ge 0$ be a
sequence of DiPerna-Lions renormalized solutions to the scaled
Vlasov-Maxwell-Boltzmann initial-value problem \eqref{me31}
with
$$G_\i(0,x,\xi)=G^0_\i(x,\xi)\ge 0, \quad E_\i(0, x)=E^0_\i(x),\quad B_\i(0,x)=B^0_\i(x).$$

A \textit{Renormalized Solution Relative to $M$} of \eqref{me3} is
a triplet $(F_\i, E_\i, B_\i)$ such that  $$F_\i\in C(\R_+;
L^1_{loc}(R^3; L^1(\R^3))),\qquad E_\i, \; B_\i\in C_w(\R_+; L^2(\R^3)), $$
and satisfies
\begin{equation}\label{m291}
\Gamma'\left(\f{F_\i}{M}\right)\mathcal{Q}(F_\i, F_\i)\in
L^1_{loc}(\R_+\times\R^3\times\T)
\end{equation}
for all $\Gamma\in C^1(\R_+)$ such that
\begin{equation}\label{m292}
\Gamma(0)=0,\quad\textrm{and}\quad z\mapsto
(1+z)\Gamma'(z)\quad\textrm{is bounded on}\quad \R_+,
\end{equation}
has finite relative entropy for all positive time:
\begin{equation}\label{m293}
\mathcal{H}_\i(t)+\f{1}{\i}\int_0^t\mathcal{R}(G(s))ds\le
\mathcal{H}_\i(0),
\end{equation}
and finally satisfies
\begin{equation}\label{m294}
\begin{split}
&\int_0^\infty\int_\T\int_{\R^3}\Gamma\left(\f{F_\i}{M}\right)\left(\partial_t\chi+\f{1}{\i}\xi\cdot\nabla_x\chi\right)Md\xi
dxdt\\
&\qquad+e\i\int_0^\infty\int_\T\int_{\R^3}\Gamma\left(\f{F_\i}{M}\right)\left(\i
E_\i+\xi\times
B_\i\right)\cdot\nabla_\xi\chi Md\xi dxdt\\
&\quad\quad-e\i\int_0^\infty\int_\T\int_{\R^3}\Gamma\left(\f{F_\i}{M}\right)\i
E_\i\cdot\xi
\chi Md\xi dxdt\\
&\qquad+e\i\int_0^\infty\int_\T\int_{\R^3}\Gamma'\left(\f{F_\i}{M}\right)\i
E_\i\cdot\xi F_\i \chi d\xi dxdt\\
&\qquad+\f{1}{\i^2}\int_0^\infty\int_\T\int_{\R^3}\Gamma'\left(\f{F_\i}{M}\right)\mathcal{Q}(F_\i,
F_\i)\chi Md\xi dxdt\\
&=0
\end{split}
\end{equation}
for each test function $\chi\in
C_0^\infty((0,\infty)\times\R^3\times\R^3)$.

Throughout the rest of this paper, we assume that the initial data
$G^0_\i$ satisfies the normalizations and the entropy bound
\begin{equation}\label{m34}
\mathcal{H}_\i(0)\le C\i^3,
\end{equation}
for some fixed $C>0$.

\subsubsection{Weak formulation of the limiting system \eqref{IM}} Inspired by \cite{ST}, for the
limiting system \eqref{IM} with mean zero initial data, the Leray
theory is set in the following Hilbert spaces of vector- and
scalar-valued functions:
\begin{align*}
&\mathbf{H}_v=\left\{w\in L^2(dx;\R^3):\quad \Dv
w=0,\quad\int_{\T}w
dx=0\right\},\\
&\mathbf{H}_s=\left\{\chi\in L^2(dx;\R):\quad\int_{\T}\chi
dx=0\right\}, \\
&\mathbf{V}_v=\left\{w\in\mathbf{H}_v:\quad \int_\T |\nabla
w|^2dx<\infty\right\},\\
&\mathbf{V}_s=\left\{\chi\in\mathbf{H}_s:\quad
\int_\T|\nabla\chi|dx<\infty\right\}.
\end{align*}

Let
$\mathbf{H}=\mathbf{H}_v\oplus\mathbf{H}_v\oplus\mathbf{H}_s$
and
$\mathbf{V}=\mathbf{V}_v\oplus\mathbf{V}_v\oplus\mathbf{V}_s$.
Leray's theory yields: given any $(\u_0, B_0, \theta_0)\in
\mathbf{H}$, there exists a $(\u, B, \theta)\in C([0,\infty);
w-\mathbf{H})\cap L^2_{loc}(0,\infty; \mathbf{V})$ which equals
initially $(\u_0, B_0, \theta_0)\in \mathbf{H}$ and satisfies the
incompressible system \eqref{IM} in the sense that,  for all
$(\phi, \psi, \chi)\in\mathbf{H}\cap C^1(\T)$,
\[
\begin{split}
&\int_\T \phi\cdot\u(t)
dx-\int_{\T}\phi\cdot\u(s)dx-\int_s^t\int_{\T}\nabla_x\phi:(\u\otimes\u)dxd\tau\\
&\quad=-\mu\int_s^t\int_{\T}\nabla_x\phi:\nabla_x\u dxd\tau-\alpha
e \int_s^t\int_{\T}E\phi dxd\tau
+\int_s^t\int_{\T}\nabla_x\phi:(B\otimes B)dxd\tau;\\
&\int_\T \psi\cdot B(t)dx-\int_{\T} \psi\cdot
B(s)dx+\int_s^t\int_{\T}E\cdot(\nabla_x\times
\psi)dxd\tau=0;\\
&\int_{\T}\chi\theta(t)dx-\int_{\T}\chi\theta(s)dx-\int_s^t\int_{\T}\nabla_x\chi\cdot(\u\theta)
dxd\tau\\
&\quad=-\kappa\int_s^t\int_{\T}\nabla_x\chi\cdot\nabla_x\theta
dxd\tau,
\end{split}
\]
for every $0\le s<t$. Moreover, $(\u, B, \theta)$ satisfies the
dissipation inequalities
\begin{subequations}\label{m295}
\begin{align}
&\int_\T\f{1}{2}\left(|\u(t)|^2+\alpha|B(t)|^2\right)+\int_0^t\mu|\nabla_x\u|^2dxds\le
\int_\T\f{1}{2}\left(|\u_0|^2+\alpha|B_0|^2\right)dx, \label{m2951}\\
&\int_\T\f{1}{2}|\theta(t)|^2dx+\int_0^t\int_\T\kappa|\nabla_x\theta|^2dxds\le
\int_\T\f{1}{2}|\theta_0|^2dx, \label{m2952}
\end{align}
\end{subequations}
for every $t>0$.

A global existence theory, similar to Leray's theory of
incompressible Navier-Stokes equations, can be established via
Garlerkin's method, the dissipation inequalities \eqref{m295} and
Ohm's law which expresses the electric field $E$ in terms of the
magnetic field and the velocity as, see \cite{BS, MG}
$$j=\sigma(E+\u\times B),$$ where $\sigma>0$ is the electrical conductivity. To obtain the dissipation inequality
\eqref{m2951}, we first multiply \eqref{IM1} by $\u$ to obtain,
using \eqref{IM2},
\begin{equation}\label{IME1}
\f{1}{2}\f{d}{dt}\|\u\|_{L^2(\T)}^2+\mu\|\nabla\u\|_{L^2(\T)}^2-\alpha
\int_\T E\cdot (\nabla\times B) dx=0.
\end{equation}
Here, we used the identity
$$B\times(\nabla\times B)\cdot\u=\f{1}{e}(B\times(\nabla\times
B))\cdot j=\f{1}{e}(B\times(\nabla\times B))\cdot (\nabla\times
B)=0,$$ according to \eqref{IM2}. Then, we multiply \eqref{IM2} by
$\alpha B$ to obtain
\begin{equation}\label{IME2}
\f{\alpha}{2}\f{d}{dt}\|B\|_{L^2(\T)}^2+\alpha\int_\T
E\cdot(\nabla\times B)=0.
\end{equation}
Adding \eqref{IME1} and \eqref{IME2}, and then integrating it over
$(0,T)$ yield the energy inequality \eqref{m2951}.

In summary, we have the following existence theory for the
incompressible system \eqref{IM}.
\begin{Proposition} For each $\u_0, B_0\in\{f\in L^2(\R^3):\Dv f=0\quad\textrm{in}\quad
\mathcal{D}'\}$ and $\t_0\in L^2(\R^3)$, there exists at least one
weak solution $(\u, B, \t)$ of \eqref{IM}-\eqref{IMIN} that
satisfies the energy inequality
\begin{equation*}
\begin{split}
&\f{1}{2}\int_{\mathcal{T}}\left(|\u(t,x)|^2+\alpha
|B(t,x)|^2+\f{5}{2}|\t(t,x)|^2\right)dx
+\int_0^t\int_{\mathcal{T}}\left(\mu|\nabla\u|^2+\f{5}{2}\kappa|\nabla\t|^2\right)dxds\\
&\quad\le \f{1}{2}\int_{\mathcal{T}}\left(|\u_0|^2+\alpha
|B_0|^2+\f{5}{2}|\t_0|^2\right)dx
\end{split}
\end{equation*}
for all $t>0$.
\end{Proposition}

\subsection{Assumptions}

In this subsection, we state our technical assumptions. To begin
with, we define
$$A(\xi)=\int_{S^2}b(\xi,\omega)d\omega.$$
Our assumptions regarding the collision kernel $b$ are stated as
follows:
\begin{itemize}
\item {\bf (H0)} $b\in L^1(B_R\times S^2)\quad \textrm{for all}\quad R\in
(0,\infty)$, where $B_R=\{z\in \R^3: |z|<R\}$, and
\begin{equation*}
\begin{cases}
b(z,w)\quad \textrm{depends only on}\quad |z|\quad
\textrm{and}\quad |(z,\omega)|,\\
(1+|z|^2)^{-1}\left(\int_{z+B_R}A(\xi)d\xi\right)\rightarrow
0,\quad \textrm{as}\quad |z|\rightarrow\infty,\quad \textrm{for
all}\quad R\in(0,\infty).
\end{cases}
\end{equation*}
\item {\bf (H1)} $\f{1}{b_\infty}\le b(z,\omega)\le b_\infty,\quad
z\in\R^3,\quad \omega\in S^2,\quad\textrm{for some}\quad
b_\infty>0$;
\end{itemize}

The assumption ({\bf{H0}}) is assumed to make possible the global
existence of renormalized solutions to the Vlasov-Maxwell-Boltzmann
equations, see \cite{DL1, HW}. The class of collision kernels
satisfying ({\bf{H0}}), ({\bf{H1}}) is not empty since it contains
at least all collision kernels of the form
$b(z,\omega)=b(|\cos(z,\omega)|)$ satisfying ({\bf{H0}}).

Next, we impose one more technical assumption on the sequence of
fluctuations $\{g_\i\}_{\{\i>0\}}$ (see \eqref{m18} below).
\begin{itemize}
\item {\bf{(H2)}} The family $(1+|\xi|^2)\f{g_\i^2}{N_\i}$ is
relatively compact in $w-L^1(dtMd\xi dx)$, where
$N_\i=1+\f{\i}{3}g_\i$.
\end{itemize}
This assumption is the same as $(A2)$ of Lions-Masmoudi \cite{ML} and similar to
$(H2)$ of \cite{BGL}, with the only difference being that we had
to add the time variable, since we are dealing with the
nonstationary case, when compared with the stationary case in
\cite{BGL}.

\subsection{Main Result}
We consider a sequence of solutions $G_\i$ to the scaled
Vlasov-Maxwell-Boltzmann equations
\begin{equation}\label{m17}
\i\partial_t G_\i+\xi\cdot\nabla_xG_\i+e\i(\i E_\i+\xi\times
B_\i)\cdot\nabla_\xi G_\i-e\i^2 E_\i\cdot \xi
G_\i=\f{1}{\i}Q(G_\i, G_\i),
\end{equation}
in the form
\begin{equation}\label{m18}
G_\i=1+\i g_\i.
\end{equation}
We expect that as $\i$ tends to zero, the leading behavior of the
fluctuations $g_\i$ is formally consistent with the incompressible
Electron-Magnetohydrodynamics-Fourier equations. Indeed, formally,
substituting \eqref{m18} into \eqref{m17}, we obtain
\begin{equation}\label{m19}
\i\partial_t g_\i+\xi\cdot\nabla_x g_\i+e\i(\i E+\xi\times
B)\cdot\nabla_\xi g_\i-e\i E_\i\cdot\xi-e\i^2E\cdot \xi
g_\i+\f{1}{\i}Lg_\i=Q(g_\i, g_\i),
\end{equation}
where $L$, the linearized collision operator, is given by
$$Lg=-2Q(1,g)=\int_{\R^3}\int_{\R^3}(g+g_*-g'-g'_*)bd\omega
M_*d\xi_*.$$

Repeated applications of the $d\mathcal{M}$-symmetries yield the
identity
\begin{equation*}
\begin{split}
<v Lg>&=\left\langle\!\!\left\langle v(g+g_*-g'-g'_*\right\rangle\!\!\right\rangle\\
&=\f{1}{4}\left\langle\!\!\left\langle
(v+v_*-v'-v'_*)(g+g_*-g'-g_*')\right\rangle\!\!\right\rangle,
\end{split}
\end{equation*}
for every $v=v(\xi)$ and $g=g(\xi)$ for which the integral makes
sense. This shows that $L$ is formally self-adjoint and has a
non-negative Hermitian form. Furthermore, using the
$d\mathcal{M}$-characterization, it can be shown that for any
$g=g(\xi)$ in the domain of $L$, the following statements are
equivalent:
\begin{subequations}\label{m21}
\begin{align}
&Lg=0;\\
&g=\alpha+\beta\cdot\xi+\f{1}{2}\gamma|\xi|^2,\qquad\textrm{for
some }\qquad (\alpha,\beta,\gamma)\in \R\times\R^3\times\R.
\end{align}
\end{subequations}
This characterizes $N(L)$, the null space of $L$, as the set
obtained by linearizing about $(\alpha, \beta, \gamma)=(0,0,0)$.
From \eqref{m19}, we deduce formally that the limit of $Lg_\i$ is
zero and it can be expected that the limit of $g_\i$ will belong
to $N(L)$. Indeed, it was proved by Grad (see \cite{GH, GL}) that
for any collision kernel $b$ satisfying ({\bf{H1}}), $L$ is a
bounded nonnegative self-adjoint Fredholm operator on $L^2(Md\xi)$
with null space
$$Ker L=\textrm{span}\{1, \xi_1, \xi_2, \xi_3, |\xi|^2\}.$$
Notice that since each entry of the tensor
$\xi\otimes\xi-\f{1}{3}|\xi|^2I$ and of the vector
$\f{1}{2}\xi(|\xi|^2-5)$ is orthogonal to $Ker L$, there exist a
unique tensor $\Phi$ and a unique vector $\Psi$ such that
\begin{align}
&L\Phi=\xi\otimes\xi-\f{1}{3}|\xi|^2I,\quad \Phi\in(Ker
L)^\bot\subset
L^2(Md\xi);\\
&L\Psi=\f{1}{2}\xi(|\xi|^2-5),\quad \Psi\in(Ker L)^\bot\subset
L^2(Md\xi).
\end{align}

Now, our main result can be stated as follows.

\begin{Theorem}\label{T1}
Under the hypotheses {\rm ({\bf{H0}})-({\bf{H2}})}, let
$G_\i(t,x,\xi)$, with the form \eqref{m18}, be a sequence of
non-negative renormalized solutions to the scaled
Vlasov-Maxwell-Boltzmann equations \eqref{me31} satisfying the
initial condition \eqref{m34}. Then,
\begin{itemize}
\item  The sequence $g_\i$ converges in the sense of distributions
and almost everywhere to a function $g$ as $\i$ tends to zero, and
$g$ is an infinitesimal Maxwellian,
\begin{equation}\label{m23}
g=h+\u\cdot\xi+\theta\left(\f{1}{2}|\xi|^2-\f{3}{2}\right),
\end{equation}
where the velocity $\u$ satisfies the incompressibility relation,
while the density and temperature functions, $h$ and $\theta$,
satisfy the Boussinesq relation:
\begin{equation}\label{m24}
\Dv\u=0,\qquad \nabla_x(h+\theta)=0.
\end{equation}
\item As $\i\rightarrow 0$, $E_\i$ and $B_\i$  converge to $E$
and $B$ in the sense of distributions and $L^\infty_t(L^2(\T))$
respectively.
\item Moreover, the functions $h$, $\u$, $\theta$, $B$, and $E$ are
weak solutions of \eqref{IM} with
\begin{equation}\label{m28}
\mu=\f{1}{10}<\Phi:L\Phi>,\qquad\kappa=\f{2}{15}<\Psi\cdot L\Psi>.
\end{equation}
\end{itemize}
\end{Theorem}

\bigskip

\section{Implications of the Entropy Inequality}

In this section, we first recall some results in \cite{BGL, GL}
which were established in the greatest possible generality,  and
relied only on the \textit{a priori} estimates and in particular
have nothing to do with the equations. To this end, from now on,
we assume that the initial data $G_{\i}^0$ satisfies the entropy
bound:
\begin{equation}\label{3m2}
\i\int_{\mathcal{T}}\left<G^0_{\i}\ln
G_{\i}^{0}-G^0_{\i}+1\right>dx+\f{\i^3}{2}\int_\T(\i|E_\i^{0}|^2+|B^{0}_\i|^2)dx\le
C\i^3
\end{equation}
with $C>0$. From the relative entropy, we can obtain the uniform
bound $\|B_\i\|_{L^\infty_t(L^2(dx))}$, and hence we can assume
\begin{equation}\label{max1}
B_\i\rightarrow B\quad\textrm{weakly}^* \quad\textrm{in}\quad
L^\infty_t(L^2(dx)),
\end{equation}
with $\Dv B=0$ in $\mathcal{D}'$. Furthermore, from the relative
entropy, $\i^{\f{1}{2}}\|E_\i\|_{L^\infty_t(L^2(dx))}$ is
uniformly bounded, and hence, we can assume that
\begin{equation}\label{max2}
\i^{\f{1}{2}}E_\i\rightarrow\Omega,\quad\textrm{weakly}^*
\quad\textrm{in}\quad L^\infty_t(L^2(dx))
\end{equation}
for some function $\Omega\in L^\infty_t(L^2(dx))$. Then the results in
\cite{BGL, GL}, combining with \eqref{reg} and \eqref{3m2} imply
the following convergence.

\begin{Theorem}\label{bgl}
Under assumptions  {\rm ({\bf{H0}})-({\bf{H2}})}, let $F_\i$ be a
family of renormalized solutions to \eqref{me3} with initial data
$(F_{\i}^{0}, E_\i^{0}, B_\i^{0})$ satisfying \eqref{m34}, and
define the associated family of fluctuations by
$$g_\i=\f{F_\i-M}{\i M}.$$
Then
\begin{itemize}
\item $g_\i$ is relatively compact in $w-L^1_{loc}(dtdx;
L^1((1+|\xi|^2)Md\xi))$,  and for almost every $t\in [0,\infty)$,  $g$
satisfies
\begin{equation}\label{3m15}
\int_\T\f{1}{2}\left<g^2(t)\right>dx\le\liminf_{\i\rightarrow
0}\int\left<\f{1}{\i^2}h(\i g_\i(t))\right>dx\le C;
\end{equation}
moreover, for almost every $(t,x)$, $g(t,x,\cdot)\in N(L)$, which
means that $g$ is of the form
\begin{equation}\label{b1}
g(t,x,\cdot)=h(t,x)+\u(t,x)\cdot\xi+\t(t,x)\left(\f{1}{2}|\xi|^2-\f{3}{2}\right),
\end{equation}
where $(h,\u,\t)\in L^\infty(dt; L^2(dx; \R\times \R^3\times\R))$.
\item the rescaled collision integrands
\begin{equation}\label{b4}
q_\i=\f{1}{\i^2}(G'_\i G'_{\i *}-G_\i G_{\i *})
\end{equation}
satisfy that 
$\gamma(G_\i)q_\i$ is relatively compact in
$w-L^1_{loc}(dtdx; L^1((1+|\xi|^2)d\mathcal{M}));$   furthermore, any of the limit points
$q$ of $\gamma(G_\i)q_\i$ as $\i\rightarrow 0$ satisfies the
$d\mathcal{M}$-symmetry relations
\begin{equation}\label{b5}
\left\langle\!\left\langle\phi(\xi)q\right\rangle\!\right\rangle=\left\langle\!\!\!\left\langle\f{1}{4}(\phi+\phi_*-\phi'-\phi'_*)q
\right\rangle\!\!\!\right\rangle,
\end{equation}
and, $q\in L^2(0,T; L^2(d\mathcal{M} dx))$.
\item for any subsequence $\i_n\rightarrow 0$ such that
$$g_{\i_n}\rightarrow g,\quad\textrm{and}\quad
\gamma(G_{\i_n})q_{\i_n}\rightarrow q$$ in $w-L^1_{loc}(dtdx;
L^1((1+|\xi|^2)Md\xi))$ and in $w-L^1_{loc}(dtdx;
L^1((1+|\xi|^2)Md\mathcal{M}))$ respectively.
\item denoting $N_\i=\f{2}{3}+\f{1}{3}G_\i$, then $\f{g_\i}{N_\i}$ is
bounded in $L^\infty_t(L^2(Md\xi dx))$ and  relatively compact in $w-L^1_{loc}(dtdx; L^1((1+|\xi|^2)d\mathcal{M}))$.
\end{itemize}
\end{Theorem}

Weak compactness statements regarding $g_\i$ and $q_\i$ result in
the following bound for their limits.

\begin{Lemma}\label{bgl1}
Under the same conditions as Theorem \ref{bgl}, for almost every
$t\in[0,\infty)$ the function $g$ and $q$ satisfy
\begin{equation}\label{3m171}
\begin{split}
&\int_\T\f{1}{2}\left<g^2(t)\right>dx+\f{1}{2}\int_\T(|\Omega|^2+|B|^2)dx+\int_0^t\int_\T\f{1}{4}\left\langle\!\!\left\langle q^2\,\right\rangle\!\!\right\rangle dxds\\
&\le\liminf_{\i\rightarrow 0}\int_\T\left<\f{1}{\i^2}h(\i
g^0_\i)\right>dx\le C.
\end{split}
\end{equation}
\end{Lemma}
\begin{proof}
Taking the $\liminf$ on the both sides of the entropy inequality
\eqref{reg}, we obtain
\begin{equation}\label{bgl1111}
\begin{split}
&\liminf_{\i\rightarrow 0}\int_{\mathcal{T}}\left<\f{1}{\i^2}h(\i
g_{\i}(t))\right>dx+\f{1}{2}\liminf_{\i\rightarrow
0}\int_{\T}(\i|E_\i|^2+|B_\i|^2)dx\\&\quad+\liminf_{\i\rightarrow
0}\int_0^t\int_{\mathcal{T}}\f{1}{4}\left\langle\!\!\!\left\langle\f{1}{\i^4}r\left(\f{\i^2
q_\i}{G_{*\i
G_\i}}\right)G_{*\i}G_{\i}\right\rangle\!\!\!\right\rangle dxds
\\&\le \liminf_{\i\rightarrow
0}\left(\int_{\mathcal{T}}\left<\f{1}{\i^2}h(\i
G^0_{\i}(t))\right>dx+\f{1}{2}\int_{\T}(\i|E^{0}_\i|^2+|B^{0}_\i|^2)dx\right)\\&\le
C.
\end{split}
\end{equation}

Due to the lower semi-continuity of the weak convergence, we
deduce that
\begin{equation}\label{bgl1112}
\int_{\T}(|\Omega|^2+|B|^2)dx\le\liminf_{\i\rightarrow
0}\int_{\T}(\i|E_\i|^2+|B_\i|^2)dx,
\end{equation}
while, from the second assertion of Proposition 3.1 in \cite{BGL},
we have
\begin{equation}\label{bgl1113}
\begin{split}
&\int_\T\f{1}{2}\left<g^2(t)\right>dx+\int_0^t\int_\T\f{1}{4}\left\langle\!\!\left\langle
q^2\right\rangle\!\!\right\rangle dxds\\&\le\liminf_{\i\rightarrow
0}\int_{\mathcal{T}}\left<\f{1}{\i^2}h(\i
g_{\i}(t))\right>dx+\liminf_{\i\rightarrow
0}\int_0^t\int_{\mathcal{T}}\f{1}{4}\left\langle\!\!\!\left\langle\f{1}{\i^4}r\left(\f{\i^2
q_\i}{G_{*\i
G_\i}}\right)G_{*\i}G_{\i}\right\rangle\!\!\!\right\rangle dxds.
\end{split}
\end{equation}
Substituting \eqref{bgl1112} and \eqref{bgl1113} back into
\eqref{bgl1111}, we finish the proof of \eqref{3m171}.
\end{proof}

To better understand the behavior of the fluctuation
$\{g_\i\}_{\{\i>0\}}$, as in \cite{GL} we introduce a class of
bump functions
\begin{equation}\label{3m18}
\Upsilon=\left\{\gamma:\R_+\rightarrow[0,1]|\gamma\in C^1,\quad
\gamma\left(\left[\f{3}{4},\f{5}{4}\right]\right)=\{1\},\quad
\textrm{supp}\gamma\subset\left[\f{1}{2},\f{3}{2}\right]\right\}.
\end{equation}
We decompose $g_\i$ as
\begin{equation}\label{3m19}
g_\i=g_\i^b+\i g_\i^c
\end{equation}
with
$$g_\i^b=\f{1}{\i}(G_\i-1)\gamma(G_\i),\quad
g_\i^c=\f{1}{\i^2}(G_\i-1)(1-\gamma(G_\i)),$$ where
$\gamma\in\Upsilon$. The following entropy controls (Proposition
2.1 and Proposition 2.7 in \cite{GL}) will be very useful:

\begin{Lemma}[{\bf Entropy controls}]\label{ec} Assume that the bump
function $\gamma\in\Upsilon$ as in \eqref{3m18}. The relative
fluctuation $g_\i$ of the density satisfies the following
estimates:
\begin{itemize}
\item $\i|g_\i^b|\le\f{1}{2}$ and $$g_\i^b=O(1)\quad\textrm{in}\quad L^\infty_t(L^2(Md\xi
dx));$$
\item $(1-\gamma(G_\i))\le 4\i^2|g_\i^c|$, which implies that $\f{1}{\i}(1-\gamma(G_\i))\le 2|g_\i^c|^{\f{1}{2}}$, and
$$g_\i^c=O(1)\quad\textrm{in}\quad L^\infty_t(L^1(Md\xi
dx));$$
\item $(1-\gamma(G_\i))G_\i\le 5\i^2|g_\i^c|$, and $(1-\gamma(G_\i))\le 4\i^2|g_\i^c|$.
\end{itemize}
\end{Lemma}

\bigskip

\section{Implications of the Maxwell Equations}

For the asymptotic behavior of the solutions under the hypothesis
$\mathcal{H}_\i(0)\le C\i^3$, one of the difficulties when we deal
with the magnetic field and the electric field comes from the fact
that the relative entropy does not provide useful information on
the electric field $E_\i$ due to the $\i$ in the front of the
electric field in the definition of the relative entropy
$\mathcal{H}_\i$. Fortunately, the uniform estimate from the
relative entropy is enough to ensure that $\Omega=0$. Indeed, from
\eqref{max2},
$$\i \f{\partial{E_\i}}{\partial t}\rightarrow 0,\quad
\textrm{in}\quad \mathcal{D}'(\R_+\times \R^3).$$ Next, since
$g_\i$ converges to $g$ in $w-L^1_{loc}(dtdx; L^1((1+|\xi|^2)Md\xi
dx))$, by the Cauchy-Schwarz inequality, we can deduce that $g_\i$
converges to $g$ in $w-L^1_{loc}(dtdx; L^1(|\xi|Md\xi dx))$. Due
to the fact $<\xi>=0$, $\f{j_\i}{\i}=<\xi g_\i>$, $\f{j_\i}{\i}$
converges to $j$ in $w-L^1_{loc}(dtdx)$. Then we take the limit as
$\i\rightarrow 0$ in the equation \eqref{me32} to get
\begin{align}\label{max3}
\nabla\times B=j
\end{align}
in the sense of distributions. Furthermore,
\begin{equation*}
\begin{split}
\left\|\f{j_\i}{\i}\right\|_{L^\infty_t(L^2(\T))}&=\left\|<\xi
g_\i>\right\|_{L^\infty_t(L^2(\T))}\\
&\le \|g_\i\|_{L^\infty_T(L^2(Md\xi
dx))}\left<|\xi|^2\right>^{\f{1}{2}}\\
&<\infty.
\end{split}
\end{equation*}
This implies that $\f{j_i}{\i}$ converges weakly$^*$ to $j$ in
$L^\infty_t(L^2(\T))$.

On the other hand, for the electric field $E_\i$, we have
\begin{Lemma}\label{max2l}
The family $\{E_\i\}_{\{\i>0\}}$ formally satisfies
\begin{equation}\label{4m52}
\begin{split}
E_\i&=\partial_t(\i E_\i\times B_\i)-(\nabla\times B_\i)\times
B_\i+\f{j_\i}{\i}\times B_\i+\i\Dv(E_\i\otimes
E_\i)\\&\quad-\i\f{1}{2}\nabla|E_\i|^2-\i E_\i\int_{\R^3}g_\i M
d\xi,
\end{split}
\end{equation}
in the sense of distributions. Hence, $\{E_\i\}_{\i>0}$ is
uniformly bounded in $(W^{1,\infty}_0((0,T)\times\T))'$.
\end{Lemma}
\begin{proof}
 Indeed, multiplying \eqref{me32} by $B_\i$, multiplying
\eqref{me33} by $\i E_\i$, and adding them together to yield
\begin{equation}\label{4m53}
\begin{split}
\partial_t(\i E_\i\times B_\i)-(\nabla\times B_\i)\times B_\i+\i
E_\i\times (\nabla\times E_\i)=-\f{j_\i}{\i}\times B_\i.
\end{split}
\end{equation}
Note that
\begin{equation}\label{EI}
E\Dv E+(\nabla \times E)\times E=\Dv(E\otimes
E)-\f{1}{2}\nabla|E|^2.
\end{equation}
The identity \eqref{4m53} can be rewritten as, using \eqref{me34}
\begin{equation}\label{4m54}
\begin{split}
E_\i\r_\i&=\i E_\i\Dv E_\i\\&=\partial_t(\i E_\i\times
B_\i)-(\nabla\times B_\i)\times B_\i+\f{j_\i}{\i}\times
B_\i+\i\Dv(E_\i\otimes E_\i)-\i\f{1}{2}\nabla|E_\i|^2.
\end{split}
\end{equation}
Because $$\r_\i=\int_{\R^3}(1+\i g_\i)M d\xi=\int_{\R^3}M
d\xi+\i\int_{\R^3}g_\i M d\xi=1+\i\int_{\R^3}g_\i M d\xi,$$ one
obtains, according to \eqref{4m54},
\begin{equation}\label{4m55}
\begin{split}
E_\i&=\partial_t(\i E_\i\times B_\i)-(\nabla\times B_\i)\times
B_\i+\f{j_\i}{\i}\times B_\i+\i\Dv(E_\i\otimes
E_\i)\\&\quad-\i\f{1}{2}\nabla|E_\i|^2-\i E_\i\int_{\R^3}g_\i M
d\xi.
\end{split}
\end{equation}

Next, due to the uniform bounds
$$\|\sqrt{\i}E_\i\|_{L^\infty(0,T; L^2(\R^3))}\le C,\qquad \|B_\i\|_{L^\infty(0,T; L^2(\R^3))}\le
C,$$ we have $$\partial_t(\i E_\i\times B_\i)\rightarrow 0$$ in
$(W^{1,\infty}((0,T)\times\T))'$ as $\i\rightarrow 0$, and
$-(\nabla\times B_\i)\times B_\i+\f{j_\i}{\i}\times B_\i$ is
uniformly bounded in $(W^{1,\infty}((0,T)\times\T))'$ by using the
identity \eqref{EI} for $B$.

Also, we can control the term $\i E_\i\int_{\R^3}g_\i M d\xi$ as
follows
\begin{equation*}
\begin{split}
&\left\|\i E_\i\int_{\R^3}g_\i M d\xi\right\|_{L^1((0,T)\times
\T)}\\&\quad\le \sqrt{\i}\|\sqrt{\i}E_\i\|_{L^2((0,T)\times
\T)}\left(\int_{\R^3}Md\xi\right)^{\f{1}{2}}\|<g_\i^2>\|^{\f{1}{2}}_{L^1((0,T)\times
\T)}\\
&\quad\le C \sqrt{\i}\rightarrow 0
\end{split}
\end{equation*}
as $\i\rightarrow 0$. Hence, according to \eqref{4m55}, we deduce
that $\{E_\i\}_{\i>0}$ is uniformly bounded in
$(W^{1,\infty}_0((0,T)\times\T))'$.
\end{proof}

As a direct consequence of Lemma \ref{max2l}, we have
\begin{Lemma}\label{maxl}
$E_\i\rightarrow E\quad\textrm{weakly in}\quad (W^{2,p}_0)'$, for
some function $E\in (W^{2,p}_0)'$ with $p>4$, and $(E, B)$
satisfies
\begin{equation}\label{max4}
\partial_t B+\nabla\times E=0
\end{equation} in $(W^{2,p}_0)'$.
\end{Lemma}

\begin{proof}
Indeed, the uniform bound on $E_\i$ in $(W^{1,\infty}_0)'$ and the Sobolev embedding
$$W^{2,p}_0((0,T)\times\T)\hookrightarrow W^{1,\infty}_0((0,T)\times\T)$$
 for any $4<p<\infty$ imply that $E_\i$ is uniformly bounded in $(W^{2,p}_0((0,T)\times\T))'$ and
hence is weakly convergent in $(W^{2,p}_0((0,T)\times\T))'$ since
$(W^{2,p}_0((0,T)\times\T))'$ with $4<p<\infty$ is a reflexive
space.

Next, since
$$\f{\partial B_\i}{\partial t}+\nabla\times E_\i=0,$$ holds in
$\mathcal{D}'(\R_+\times\T)$, we take an arbitrarily test function
$\phi\in C_0^\infty(\R_+\times \R^3)$ to obtain
\begin{equation}\label{max5}
-\int_0^t\int_{\T} B_\i \cdot\f{\partial \phi}{\partial t}
dxds+\int_0^t\int_{\T} E_\i \cdot\nabla\times \phi dxds=0.
\end{equation}
Hence, from \eqref{max5}, we obtain
\begin{equation}\label{max6}
\begin{split}
\int_0^t\int_{\T} E\cdot\nabla\times \phi dxds
&=\lim_{\i\rightarrow 0}\int_0^t\int_{\T} E_\i\cdot\nabla\times \phi dxds
=\lim_{\i\rightarrow 0}\int_0^t\int_{\T} B_\i \cdot\f{\partial \phi}{\partial t} dxds\\
&= \int_0^t\int_{\T} B \cdot\f{\partial \phi}{\partial t} dxds
=-\int_0^t\int_{\T} \phi \cdot\f{\partial B}{\partial t} dxds
\end{split}
\end{equation}
Hence, from \eqref{max6}, we deduce that the limits $(E,B)$
satisfy \eqref{max4}.
\end{proof}

Observe that, since $E_\i$ is convergent at least in the sense of
distributions, we can conclude that $\Omega=0$.

\bigskip

\section{Vanishing of Conservation Defects}

Before stating the main result of the present section, we
introduce a new class of bump functions as in \cite{GL}. For each
$C>0$, set
$$\Upsilon_C=\left\{\gamma\in\Upsilon:\; \|\gamma'\|_{L^\infty}\le C\right\}.$$
Consider the transformation $\mathfrak{T}$ defined by
$\mathfrak{T}\gamma=1-(1-\gamma)^2$; clearly $\mathfrak{T}$ maps
$\Upsilon_C$ into $\Upsilon_{2C}$. Define
\begin{equation}\label{4m1}
\tilde{\Upsilon}=\mathfrak{T}\Upsilon_8\subset\Upsilon_{16},
\end{equation}
and notice that $\tilde{\Upsilon}\neq \emptyset$ since
$\Upsilon_8\neq \emptyset$. For each $\gamma\in\tilde{\Upsilon}$,
define
\begin{equation}\label{4m2}
\hat{\gamma}(z)=\gamma(z)+(z-1)\f{d\gamma}{dz}.
\end{equation}
Notice that
\begin{equation}\label{4m3}
\textrm{supp}\hat{\gamma}\subset\left[\f{1}{2},\f{3}{2}\right],\qquad\hat{\gamma}\left(\left[\f{3}{4},\f{5}{4}\right]\right)=\{1\}.
\end{equation}
On the other hand, let $\tilde{\gamma}\in\Upsilon_8$ be such that
$\gamma=\mathfrak{T}\tilde{\gamma}$. One has
$$1-\hat{\gamma}(z)=(1-\tilde{\gamma})\left[(1-\tilde{\gamma})-2(z-1)\f{d\tilde{\gamma}}{dz}\right],\qquad
z\ge 0$$ so that
\begin{equation}\label{4m4}
|1-\hat{\gamma}|\le 9(1-\tilde{\gamma}),\qquad z\ge 0.
\end{equation}

\begin{Theorem}[{\bf Vanishing of conservation defects}] \label{vcd} Let
$\gamma\in\tilde{\Upsilon}$, and denote by $\eta\equiv \eta(\xi)$
any collision invariant (i.e. $\eta(\xi)=1\quad\textrm{or}\quad
\eta(\xi)=\xi_1,...,\xi_3\quad\textrm{or else}\quad
\eta(\xi)=|\xi|^2$) or any linear combination thereof.  Then
\begin{equation}\label{4m5}
\partial_t\left<\eta g_\i^b\right>+\f{1}{\i}\nabla_x\cdot\left<\xi\eta
g_\i^b\right>+e B_\i\cdot\left<\xi\times\nabla_\xi\eta
g_\i^b\right>-eE_\i\cdot\left<\xi\eta\right>\rightarrow 0,
\end{equation}
in $L^1_{loc}(\R_+\times \T)$ as $\i\rightarrow 0$.
\end{Theorem}

\begin{proof}
We begin with the renormalized form \eqref{m294} of the
Vlasov-Maxwell-Boltzmann equations \eqref{me3} with
$\Gamma(z)=(z-1)\gamma(z)$
\begin{equation}\label{4m6}
\begin{split}
&\left(\partial_t+\f{1}{\i}\xi\cdot\nabla_x\right)(M g_\i^b)+e(\i
E_\i+\xi\times B_\i)\cdot\nabla_\xi(M g_\i^b)\\&\quad+e\i
E_\i\cdot\xi M g_\i^b
-e\left(\gamma(G_\i)+(G_\i-1)\f{d\gamma}{dz}(G_\i)\right)E_\i\cdot\xi F_\i \\
&=\f{1}{\i^3}\int_{S^2}\int_{\R^3}({F'}_\i {F'}_{\i *}-F_\i F_{\i*})\left(\gamma(G_\i)+(G_\i-1)\f{d\gamma}{dz}(G_\i)\right)b
d\omega M_*d\xi_*.
\end{split}
\end{equation}
Here, we used the decomposition \eqref{3m19}.
From \eqref{4m6}, we deduce that
\begin{equation}\label{4m9}
\begin{split}
&\partial_t\left<\eta
g_\i^b\right>+\f{1}{\i}\nabla_x\cdot\left<\xi\eta
g_\i^b\right>+e\int_{\R^3}(\i E_\i+\xi\times
B_\i)\cdot\nabla_\xi(Mg_\i^b)\eta d\xi\\&\quad+e\int_{\R^3}\i
E_\i\cdot\xi M g_\i^b \eta d\xi-e\int_{\R^3}\hat{\gamma}_\i
E_\i\cdot\xi F_\i \eta d\xi\\&\quad=\f{1}{\i}\ll
q_\i\hat{\gamma}_\i \eta\gg,
\end{split}
\end{equation}
where
$$\hat{\gamma}_\i=\hat{\gamma}(G_\i),$$
and the function $\hat{\gamma}$ is defined in terms of $\gamma$ by
\eqref{4m2}

Observing that
$$(X\times Y)\cdot Z=Y\cdot(Z\times X)=X\cdot(Y\times Z),$$
we have
\begin{equation}\label{4m10}
\begin{split}
&\int_{\R^3}(\i E_\i+\xi\times B_\i)\cdot\nabla_\xi(Mg_\i^b)\eta d\xi\\
&=-\left(\i E_\i\cdot\left<\nabla_\xi\eta
g_\i^b\right>+\int_{\R^3}(\xi\times B_\i)\cdot\nabla_\xi\eta g_\i^b Md\xi\right)\\
&=-\left(\i E_\i\cdot\left<\nabla_\xi\eta
g_\i^b\right>-B_\i\cdot\left<\xi\times\nabla_\xi\eta
g_\i^b\right>\right).
\end{split}
\end{equation}
Notice that following the same line of the argument of Proposition
4.1 in \cite{GL}, it can be shown that
\begin{equation}\label{4m8}
\left\|\f{1}{\i}\ll q_\i\hat{\gamma}_\i
\eta\gg\right\|_{L^1_{loc}(\R_+\times\T)}\rightarrow 0
\end{equation}
as $\i\rightarrow 0$.

In order to estimate the $L^1$-norm of the conservation defects,
for the last two terms on the left-hand side of \eqref{4m9}, we
claim
\begin{equation}\label{4m71}
\left\|\i E_\i\cdot\left<\nabla_\xi\eta
g_\i^b\right>\right\|_{L^1_{loc}(\R_+\times\T)}\rightarrow 0;
\end{equation}
\begin{equation}\label{4m7}
\left\|\int_{\R^3}\i E_\i\cdot\xi M g_\i^b \eta
d\xi\right\|_{L^1_{loc}(\R_+\times\T)}\rightarrow 0;
\end{equation}
and
\begin{equation}\label{4m711}
\left\|\int_{\R^3}\hat{\gamma}_\i E_\i\cdot\xi F_\i \eta
d\xi-\int_{\R^3} E_\i\cdot\xi\eta M d\xi
\right\|_{L^1_{loc}(\R_+\times\T)}\rightarrow 0
\end{equation}
as $\i\rightarrow 0$.
Indeed, using the elementary bounds
\begin{equation}
|\hat{\gamma}_\i|\le 9,\qquad |1-\hat{\gamma}_\i|\le 9,\qquad 0\le
G_\i|\hat{\gamma}_\i|\le \f{27}{2},
\end{equation}
for the inequality \eqref{4m71}, we have,
\begin{equation*}
\begin{split}
\left|\left<\nabla_\xi\eta
g_\i^b\right>\right|\le\left(\int_{\R^3}(\nabla_\xi\eta)^2Md\xi
\right)^{\f{1}{2}}\left(\int_{\R^3}(g_\i^b)^2
Md\xi\right)^{\f{1}{2}}\le C\left(\int_{\R^3}(g_\i^b)^2
Md\xi\right)^{\f{1}{2}},
\end{split}
\end{equation*}
since $$\int_{\R^3}(\nabla_\xi\eta)^2 Md\xi\le C$$ for all
$\eta\in N(L)$ and where $C$ is a positive constant. Hence, by the
Cauchy-Schwarz inequality and the first statement in Lemma
\ref{ec}, one has
\begin{equation*}
\begin{split}
\left\|\i E_\i\cdot\left<\nabla_\xi\eta
g_\i^b\right>\right\|_{L^1_{loc}(\R_+\times\T)}&\le C\left\|\i
|E_\i|\left(\int_{\R^3}(g_\i^b)^2
Md\xi\right)^{\f{1}{2}}\right\|_{L^1_{loc}(\R_+\times\T)}\\
&\le
C\i^{\f{1}{2}}\|\i^{\f{1}{2}}E_\i\|_{L^\infty_t(L^2(\T))}^{\f{1}{2}}\|g_\i^b\|_{L^\infty_t(L^2(Mdxd\xi))}^{\f{1}{2}}\\
&\le C\i^{\f{1}{2}}\rightarrow 0,
\end{split}
\end{equation*}
as $\i\rightarrow 0$.
Similarly, for the inequality \eqref{4m7}, we have,
\begin{equation*}
\begin{split}
\left|\int_{\R^3}\xi\eta g_\i^b
Md\xi\right|\le\left(\int_{\R^3}(\xi\eta)^2Md\xi
>\right)^{\f{1}{2}}\left(\int_{\R^3}(g_\i^b)^2 Md\xi\right)^{\f{1}{2}}\le C\left(\int_{\R^3}(g_\i^b)^2
Md\xi\right)^{\f{1}{2}},
\end{split}
\end{equation*}
since $$\int_{\R^3}(\xi\eta)^2 Md\xi\le C$$ for all $\eta\in N(L)$
, where $C$ is a positive constant. Hence, by the Cauchy-Schwarz
inequality and the first statement in Lemma \ref{ec}, one has
\begin{equation*}
\begin{split}
\left\|\int_{\R^3}\i E_\i\cdot\xi\eta g_\i^b Md\xi
\right\|_{L^1_{loc}(\R_+\times\T)}&\le C\left\|\i
|E_\i|\left(\int_{\R^3}(g_\i^b)^2
Md\xi\right)^{\f{1}{2}}\right\|_{L^1_{loc}(\R_+\times\T)}\\
&\le
C\i^{\f{1}{2}}\|\i^{\f{1}{2}}E_\i\|_{L^\infty_t(L^2(\T))}^{\f{1}{2}}\|g_\i^b\|_{L^\infty_t(L^2(Mdxd\xi))}^{\f{1}{2}}\\
&\le C\i^{\f{1}{2}}\rightarrow 0,
\end{split}
\end{equation*}
as $\i\rightarrow 0$.

It remains to deal with \eqref{4m711}. To this end, we rewrite
\begin{equation}\label{4m7111}
\begin{split}
\int_{\R^3}\hat{\gamma}_\i E_\i\cdot\xi F_\i \eta
d\xi-\int_{\R^3}E_\i\cdot\xi\eta
Md\xi&=\int_{\R^3}(\hat{\gamma}_\i-1) E_\i\cdot\xi F_\i \eta
d\xi+\i\int_{\R^3}E_\i\cdot\xi \eta g_\i Md\xi\\
&=I_1+I_2.
\end{split}
\end{equation}
Notice that from \eqref{4m4}, we have
\begin{equation*}
\begin{split}
\left|\hat{\gamma}_\i-1\right|\le 9(1-\tilde{\gamma}(G_\i))\le
9(1-\tilde{\gamma}(G_\i))^{\f{1}{2}}
\end{split}
\end{equation*}
for some $\tilde{\gamma}\in \Upsilon_8$ and hence we can control
$I_1$ as, using $F_\i=MG_\i$, Lemma \ref{ec} and the fact $0\le
1-\tilde{\gamma}(G_\i)\le 1$,
\begin{equation}\label{4m7112}
\begin{split}
\|I_1\|_{L^1_{loc}(\R_+\times\T)}&\le
9\left\|\int_{\R^3}|E_\i||\xi\eta|
|1-\tilde{\gamma}(G_\i)|^{\f{1}{2}}G_\i
Md\xi\right\|_{L^1_{loc}(\R_+\times\T)}\\
&\le 9\left\|\int_{\R^3}|E_\i||\xi\eta|
|1-\tilde{\gamma}(G_\i)|^{\f{1}{2}}
Md\xi\right\|_{L^1_{loc}(\R_+\times\T)}\\
&\quad+9\i\left\|\int_{\R^3}|E_\i||\xi\eta|
|1-\tilde{\gamma}(G_\i)|^{\f{1}{2}}g_\i
Md\xi\right\|_{L^1_{loc}(\R_+\times\T)}\\
&\le 18\sqrt{\i}\left\|\sqrt{\i}\int_{\R^3}|E_\i||\xi\eta|
|g_\i^c|^{\f{1}{2}}
Md\xi\right\|_{L^1_{loc}(\R_+\times\T)}\\
&\quad +9\i\left\|\int_{\R^3}|E_\i||\xi\eta| g_\i
Md\xi\right\|_{L^1_{loc}(\R_+\times\T)}\\
&\le
18\sqrt{\i}\|\sqrt{\i}E_\i\|_{L^2_{loc}(\R_+\times\T)}\left<|\xi\eta|^2\right>^{\f{1}{2}}\left\|
|g_\i^c|\right\|_{L^1_{loc}(\R_+\times\T;
L^1(Md\xi))}^{\f{1}{2}}\\
&\quad
+9\i\|E_\i\|_{L^2_{loc}(\R_+\times\T)}\left<|\xi\eta|^2\right>^{\f{1}{2}}\left\|
g_\i\right\|_{L^2_{loc}(\R_+\times\T;
L^2(Md\xi))}\\
&\le C\sqrt{\i}+C\i\rightarrow 0
\end{split}
\end{equation}
as $\i\rightarrow 0$.
For $I_2$, we have
\begin{equation}\label{4m7113}
\begin{split}
\|I_2\|_{L^1_{loc}(\R_+\times\T)}&\le
\sqrt{\i}\|\sqrt{\i}E_\i\|_{L^2_{loc}(\R_+\times\T)}\left<|\xi\eta|\right>^{\f{1}{2}}\|g_\i\|_{L^2_{loc}(\R_+;
L^2(Md\xi dx))}\\
&\le C\i\rightarrow 0
\end{split}
\end{equation}
as $\i\rightarrow 0$.
Adding \eqref{4m7111}, \eqref{4m7112} and \eqref{4m7113} together
gives \eqref{4m711}.
Combining \eqref{4m9}--\eqref{4m711}, the proof of \eqref{4m5} is
finished.
\end{proof}

\begin{Remark}
According to Theorem \ref{vcd}, if $\eta=1$ or $\eta=|\xi|^2$,
then the last term on the left hand side of \eqref{4m5} will
vanish; that is,
$$E_\i\cdot\left<\xi\right>=E_\i\cdot\left<\xi|\xi|^2\right>=0,$$
because
$$\left<\xi\right>=\left<\xi|\xi|^2\right>=0.$$ This implies that the term
$E_\i\cdot<\xi\eta>$ will only possibly appear in the conservation
law of momentum. Hence,
\begin{equation}\label{4m51}
\partial_t\left<g_\i^b \xi_k\right>+\f{1}{\i}\nabla_x\cdot\left<\xi \xi_k
g_\i^b\right>+ eB_\i\cdot\left<\xi\times\nabla_\xi\xi_k
g_\i^b\right>-\alpha e (E_\i)_k\rightarrow 0,
\end{equation}
in $L^1_{loc}(\R_+\times\T)$ for all $1\le k\le 3$, since
$\left<\xi^2_k\right>=\alpha=\f{1}{3}\left<|\xi|^2\right>$.
\end{Remark}

\bigskip

\section{Proof of the Main Result: Theorem \ref{T1}}

In this section, we will finish the proof of Theorem \ref{T1} via
three steps.

\subsection{The Incompressibility and Boussinesq Relations}

Let us start with considering the renormalized form of the first
equation in \eqref{me3}:
\begin{equation}\label{4m001}
\i\partial_t h_\i+\xi\cdot\nabla_x h_\i+e\i\left(\i E_\i+\xi\times
B_\i\right)\cdot\nabla_\xi h_\i -e\i E_\i\cdot\xi \f{G_\i}{ N_\i}
=\f{1}{\i^2}\f{1}{N_\i}Q(G_\i, G_\i),
\end{equation}
where
$$h_\i=\f{3}{\i}\ln \left(1+\f{1}{3}\i g_\i\right)=\f{3}{\i}\ln N_\i.$$ Since
$h_\i$ formally behaves like $g_\i$ for small $\i$, it should be
thought of as the normalized form of the fluctuations $g_\i$. This
means that,  for every $\chi\in C^1(\mathcal{T}; L^\infty(Md\xi))$
and every $0\le s\le t<\infty$, one has,
\begin{equation}\label{4m002}
\begin{split}
&\i\int_\T <h_\i(t)\chi>dx-\i\int_\T <h_\i(s)\chi>dx-\int_s^t\int_\T <h_\i\xi\cdot\nabla_x\chi>dxd\tau\\
&\quad+e\int_s^t\int_\T \i^2 E_\i\cdot<\xi
h_\i\chi>dxd\tau-e\int_s^t\int_\T\int_{\R^3} \i(\i E_\i+\xi\times
B_\i)\cdot \nabla_\xi \chi h_\i Md\xi dxd\tau\\
&\quad-e\int_s^t\int_\T\i E_\i\cdot\left<\xi \f{G_\i}{
N_\i}\right>dxd\tau
\\&=\int_s^t\int_\T\left\langle\!\!\!\left\langle\f{q_\i}{N_\i}\chi\right\rangle\!\!\!\right\rangle dxd\tau.
\end{split}
\end{equation}
Due to the fact $$\f{G_\i}{N_\i}\le 3$$ and the entropy control
$$\|\i^{\f{1}{2}} E_\i\|_{L^\infty_t(L^2(dx))}\le C,$$ one obtains
$$\int_s^t\int_\T\i
E_\i\cdot\left<\xi \f{G_\i}{N_\i}\right>dxd\tau\rightarrow 0,$$ as
$\i\rightarrow 0$. On the other hand, since as stated in the last
statement of Theorem \ref{bgl} (cf. also Corollary 3.2 in
\cite{BGL}) that $h_\i$ has the same limit $g$ as the sequence
$g_\i$ in $w-L^2_{loc}(dt; w_L^2(Md\xi dx))$, one deduces that
$$\i\int_\T \left<h_\i(t)\chi\right>dx-\i\int_\T \left<h_\i(s)\chi\right>dx\rightarrow
0;$$
$$\int_s^t\int_\T \i^2 E_\i\cdot\left<\xi
h_\i\chi\right>dxd\tau\rightarrow 0;$$ and
$$\int_s^t\int_\T\int_{\R^3} \i(\i E_\i+\xi\times
B_\i)\cdot \nabla_\xi \chi h_\i Md\xi dxd\tau\rightarrow 0,$$ as
$\i\rightarrow 0$, thanks to the uniform bounds
$$\|\i^{\f{1}{2}} E_\i\|_{L^\infty_{\R_+}(L^2(\T))}\le C,\qquad  \|B_\i\|_{L^\infty_{\R_+}(L^2(\T))}\le
C.$$ Taking the limit in \eqref{4m002} as $\i$ tends to zero while
using Theorem \ref{bgl} to establish the limits of the terms
involving $h_\i$ and $q_\i$ respectively yields
$$-\int_s^t\int_\T
\left<g\xi\cdot\nabla_x\chi\right>dxd\tau=\int_s^t\int_\T\left\langle\!\left\langle
q\chi\right\rangle\!\right\rangle dxd\tau;$$ hence, the limiting
form of \eqref{4m001} is
\begin{equation}\label{4m003}
\xi\cdot\nabla_x g=\int\int q b(\xi_*-\xi, \omega)d\omega
M_*d\xi_*.
\end{equation}

Since $q$ is in $L^2(d\mathcal{M} dx)$, then for every
$\eta=\eta(\xi)$ in $L^2(d\mathcal{M})$, an application of the
Cauchy-Schwarz inequality shows that $\left\langle\!\left\langle
\eta q\right\rangle\!\right\rangle$ is in $L^2(dx)$. By a repeated
application of the $d\mathcal{M}$-symmetries in Theorem \ref{bgl},
one has that,  for any $\eta$ in $L^2(d\mathcal{M})$,
\begin{equation}\label{4m01}
\left\langle\!\left\langle \eta
q\right\rangle\!\right\rangle=\f{1}{4}\left\langle\!\!\left\langle
(\eta+\eta_*-\eta'_*-\eta') q\right\rangle\!\!\right\rangle.
\end{equation}
Successively apply the identity \eqref{4m01} for $\eta=1, \xi,
\f{1}{2}|\xi|^2$ and use the microscopic conservation laws
\eqref{ei} to obtain
$$\left\langle\!\left\langle q
\right\rangle\!\right\rangle=0 ,\quad\left\langle\!\left\langle
\xi
q\right\rangle\!\right\rangle=0,\quad\left\langle\!\!\!\left\langle
\f{1}{2}|\xi|^2q\right\rangle\!\!\!\right\rangle=0.$$

Since these $\eta$ are also in $L^2(M d\xi)$, it then follows from
the limiting Vlasov-Maxwell-Boltzmann equation \eqref{4m003} that
$g$ satisfies the local conservation laws of mass, momentum, and
energy:
\begin{equation}\label{4m02}
\Dv_x\left<\xi g\right>=0,\qquad\Dv_x\left<\xi\otimes\xi
g\right>=0,\qquad \Dv_x\left<\xi\f{1}{2}|\xi|^2g\right>=0.
\end{equation}

Theorem \ref{bgl} states that $g$ has the form of the
infinitesimal Maxwellian
$$g=h+\u\cdot\xi+\t\left(\f{1}{2}|\xi|^2-\f{3}{2}\right).$$
Substituting this into \eqref{4m02}, the local mass and energy
conservation laws yield the incompressibility relation for the
velocity field $\u$ while that of momentum yields the Boussinesq
relation between $h$ and $\t$:
$$\Dv_x\u=0,\qquad\nabla_x(h+\t)=0.$$

\subsection{Proof of Convergence to Incompressible
Electron-Magnetohydrodynamic-Fourier Equations}

Throughout this subsection, it is assumed that the bump function
$\gamma$ belongs to $\tilde{\Upsilon}$ (defined by \eqref{4m1}).
Using Theorem \ref{vcd}, the classical Sobolev embedding theorems,
and the continuity of pseudo-differential operators of order 0 on
$W^{s,p}$ for $1<p<\infty$, one sees that, for all $s>0$,
\begin{equation}\label{5m1}
\begin{split}
&\partial_t P\left<\xi
g_\i^b\right>+P\nabla_x\cdot\f{1}{\i}\left<\left(\xi\otimes\xi-\f{1}{3}|\xi|^2I\right)g_\i^b\right>\\&\quad+eP\left(
B_\i\cdot\left<\xi\times\nabla_\xi\eta g_\i^b\right>\right)-\alpha
eP E_\i\\&\rightarrow 0
\end{split}
\end{equation}
in $L^1_{loc}(dt; W^{-s,1}_{loc}(\R^3))$, and
\begin{equation}\label{5m2}
\begin{split}
\partial_t\left<\left(\f{1}{5}|\xi|^2-1\right)g_\i^b\right>+\nabla_x\cdot\f{1}{\i}\left<\xi\left(\f{1}{5}|\xi|^2-1\right)g_\i^b\right>
\rightarrow 0
\end{split}
\end{equation}
in $L^1_{loc}(dtdx)$ as $\i\rightarrow 0$. Here, the operator $P$
is the Leray projection, i.e. the $L^2(dx)$-orthogonal projection
on the space of divergence-free vector fields. In \eqref{5m2}, we
used
$$\xi\times\nabla_\xi\left(\f{1}{5}|\xi|^2-1\right)=0.$$

By Theorem \ref{bgl} and Proposition \ref{ec}, pick any sequence
$\i_n\rightarrow 0$ such that
\begin{align}\label{5m3}
&g_{\i_n}^b\rightarrow g\quad\textrm{in}\quad
w^*-L^\infty_t(L^2(Md\xi dx)),\\
&\gamma_{\i_n}q_{\i_n}^b\rightarrow q\quad\textrm{in}\quad
w-L^1_{loc}(L^1(dtdx; L^1((1+|\xi|^2)d\mathcal{M})).
\end{align}
In this section, we deal exclusively with such extracted
sequences, drop the index $n$ and abuse the notations $g_\i$,
$g_\i^b$, $g_\i^c$, $q_\i$ and so on to designate the subsequences
$g_{\i_n}$, $g_{\i_n}^b$, $g_{\i_n}^c$, $q_{\i_n}$. Set $\u$ and
$\t$ the limiting fluctuations of velocity and temperature fields
defined by
\begin{align}\label{5m4}
&\left<\xi g_\i^b\right>\rightarrow \u,\quad\textrm{in}\quad
w^*-L_t^\infty(L^2_x);\\
&\left<\left(\f{1}{3}|\xi|^2-1\right)g_\i^b\right>\rightarrow\t,\quad\textrm{in}\quad
w^*-L_t^\infty(L^2_x).
\end{align}

The second entropy control in Proposition \ref{ec} implies that
$g_\i^b$ and $g_\i$ have the same limit $g$ in $w-L^1_{loc}(dtdx;
L^1(Md\xi))$; hence the Boussinesq relation and the
incompressibility condition hold:
\begin{equation}\label{5m5}
\Dv_x\u=0,\qquad \t+<g>=0.
\end{equation}

Denote by $\varsigma$ either the tensor $\Phi$ or the vector
$\Psi$. Since $L$ is self-adjoint on $L^2(Md\xi)$ so that
\begin{equation}\label{5m6}
\begin{split}
\f{1}{\i}\left<(L\varsigma)g_\i^b\right>=\f{1}{\i}\left<\varsigma
(Lg_\i^b)\right>&=\f{1}{\i}\left\langle\!\!\left\langle\varsigma(g_\i^b+g_{\i
*}^b-{g_\i^b}'-{g_{\i *}^b}')\right\rangle\!\!\right\rangle\\
&=\left\langle\!\!\!\left\langle
\varsigma\left[\f{1}{\i}(g_\i^b+g_{\i *}^b-{g_\i^b}'-{g_{\i
*}^b}')+(g_\i^b g_{\i *}^b-{g_\i^b}'{g_{\i
*}^b}')\right]\right\rangle\!\!\!\right\rangle\\
&\quad+\left<\varsigma Q(g_\i^b, g_\i^b)\right>.
\end{split}
\end{equation}
The first term on the last right hand side of \eqref{5m6}
converges to the diffusion term while the second term converges to
the convection term in the incompressible system \eqref{IM}. These
limits are analyzed in the next two lemmas. The convergence to the
diffusion term is obtained by an argument that closely follows
\cite{GL}, except that the present work should pay additional
attention to the Maxwell effect. This apparently minor difference
makes our analysis slightly more difficult than that in \cite{GL}.

\begin{Lemma} Define
\begin{equation}\label{cons}
\nu=\f{1}{10}\left<\Phi:L\Phi\right>,\qquad
\kappa=\f{2}{15}\left<\Psi\cdot L\Psi\right>.
\end{equation}
Then, as $\i\rightarrow 0$,
$$\f{1}{\i}\left<(L\Phi)g_\i^b\right>-\left<\Phi Q(g_\i^b, g_\i^b)\right>\rightarrow
-\nu(\nabla_x\u+(\nabla_x\u)^\top);$$
$$\f{1}{\i}\left<(L\Psi)g_\i^b\right>-\left<\Psi Q(g_\i^b, g_\i^b)\right>\rightarrow
-\f{5}{2}\kappa\nabla_x\t$$ in $w-L^1_{loc}(dtdx)$.
\end{Lemma}

The convection term is the nonlinear part of the limiting system
and its convergence is therefore the most difficult to establish.
The analysis below rests not only on all \textit{a priori}
estimates and the arguments in \cite{GL}, but also the compactness
of the moment of $g_\i^b$ in $\xi$ which is stated in Lemma
\ref{ll12} below.

\begin{Lemma}
The following convergence hold in the sense of distributions on
$\R_+\times \R^3$:
$$P\nabla_x\cdot\left<\Phi Q(g_\i^b, g_\i^b)\right>\rightarrow
P\nabla_x\cdot(\u\otimes\u),$$
$$\nabla_x\left<\Psi Q(g_\i^b,
g_\i^b)\right>\rightarrow\f{5}{2}\nabla_x\cdot(\u\t),$$ as
$\i\rightarrow 0$.
\end{Lemma}

\bigskip

\subsection{The Lorentz Force Term}

The key result of this subsection is to deal with the convergence
of the Lorentz force term. To this end, we first state the
following compactness about the moment of $g_\i$ in $\xi$.

\begin{Lemma}\label{ll12}
Let $\gamma\in \Upsilon$ be the same as in \eqref{3m18} and the
hypothesis {\rm ({\bf H2})} hold. Then, the family $g_\i^b$ has the following
property:
for each sequence $\i_n\rightarrow 0$, each function
$\chi=\chi(\xi)$ such that $\f{|\chi(\xi)|}{1+|\xi|^2}\rightarrow
0$ as $|\xi|\rightarrow \infty$, each $T>0$, there exists a
function $\eta: \R_+\mapsto \R_+$ such that $\lim_{z\rightarrow
0^+}\eta(z)=0$ and
\begin{equation*}
\int_0^T\int_{\T}\left|\left<g_{\i_n}^b\chi\right>(t,x+y)-\left<g_{\i_n}^b\chi\right>(t,x)\right|^2dxdt\le\eta(|y|)
\end{equation*}
for each $y\in\R^3$ such that $|y|\le 1$, uniformly in $n$.
\end{Lemma}

\begin{proof}
For any $\gamma\in \Upsilon$, since $F_\i$ is a renormalized
solution of \eqref{me31} relatively to $M$, using the nonlinear
function $\Gamma(z)=(z-1)\gamma(z)$ in the renormalized
formulation \eqref{m294}, we obtain

\begin{equation}\label{ll210}
\begin{split}
(\i\partial_t+\xi\cdot\nabla_x)g_\i^b&=\int_{\R^3}\int_{S^2}{q_\i}
\hat{\gamma}_\i b d\omega M_*d\xi_*-e\Dv_\xi\left(\i(\i
E_\i+\xi\times B_\i)g^b_\i\right)\\
&\quad+e\hat{\gamma}_\i\i E_\i\cdot\xi G_\i,
\end{split}
\end{equation}
with $\hat{\gamma}$ defined in terms of the truncation $\gamma$ by
\eqref{4m2}. Denoting
\begin{equation*}
f\wedge L=
\begin{cases}
f,\quad\textrm{if}\quad |f|\le L;\\
L,\quad\textrm{if}\quad f\ge L;\\
-L,\quad\textrm{if}\quad f\le -L
\end{cases}
\end{equation*}
for every $L>1$, we deduce from \eqref{ll210} that
\begin{equation}\label{ll21}
\begin{split}
(\i\partial_t+\xi\cdot\nabla_x)(g_\i^b\wedge L)
&=\left(\int_{\R^3}\int_{S^2}{q_\i} \hat{\gamma}_\i b d\omega
M_*d\xi_*\right)1_{\{|g_\i^b|\le L\}}
\\&\quad-e\Dv_\xi\left(\i(\i
E_\i+\xi\times B_\i)\left(g^b_\i\wedge L\right)\right)\\
&\quad+e\hat{\gamma}_\i\i E_\i\cdot\xi G_\i 1_{\{|g_\i^b|\le L\}},
\end{split}
\end{equation}
Furthermore, for every $N>1$, we decompose $g_\i^b\wedge L$ as
$$g_\i^b\wedge L=\ov{g_\i^b}+\hat{g_\i^b},\qquad {\ov{g_\i^b}}^{0}=0,$$
with
\begin{equation}\label{1125}
\begin{split}
(\i\partial_t+\xi\cdot\nabla_x)\ov{g_\i^b}&=\left(\int_{\R^3}\int_{S^2}{q_\i}
\hat{\gamma}_\i b d\omega M_*d\xi_*\right)1_{\{|g_\i^b|\le
L\}}1_{\{|A_\i|>N\}},
\end{split}
\end{equation}
and
\begin{equation}\label{11251}
\begin{split}
(\i\partial_t+\xi\cdot\nabla_x)\hat{g_\i^b}&=\left(\int_{\R^3}\int_{S^2}{q_\i}
\hat{\gamma}_\i b d\omega M_*d\xi_*\right)1_{\{|g_\i^b|\le
L\}}1_{\{|A_\i|\le N\}}\\&\quad-e\Dv_\xi\left(\i(\i
E_\i+\xi\times B_\i)\left(g^b_\i\wedge L\right)\right)\\
&\quad+e\hat{\gamma}_\i\i E_\i\cdot\xi G_\i 1_{\{|g_\i^b|\le L\}},
\end{split}
\end{equation}
where
$$A_\i=\int_{\R^3}\int_{S^2}{q_\i}
\hat{\gamma}_\i b d\omega M_*d\xi_*.$$

{\bf{Step 1: Control of $\ov{g_\i^b}$.}} From \eqref{1125}, if we
denote
$$S_\i=\left(\int_{\R^3}\int_{S^2}{q_\i}
\hat{\gamma}_\i b d\omega
M_*d\xi_*\right)1_{\{|A_\i|>N\}}1_{\{|g_\i^b|\le L\}},$$ then we
obtain
\begin{equation}\label{11253}
\begin{split}
\ov{g_\i^b}(t,x,\xi)=\int_0^{\f{t}{\i}}S_\i(t-\i s, x-s\xi,
\xi)ds.
\end{split}
\end{equation}

Notice that, since $\left|\hat{\gamma}_\i\right|\le 9$ and $q_\i$
is weakly compact in $L^1(dtdxd\mathcal{M})$, $S_\i$ is uniformly
bounded in $L^1(dtdxMd\xi)$. Therefore,
\begin{equation}
\left\|\ov{g_\i^b}(t,x,\xi)\right\|_{L^1(dtdxMd\xi)}\le
\left\|S_\i\right\|_{L^\infty_t(L^1(dxd\mathcal{M}))}.
\end{equation}

{\bf{Step 2: Compactness of $\hat{g_\i^b}$.}} Setting
\begin{equation*}
\begin{split}
\hat{\mathcal{S}}_\i&=\left(\int_{\R^3}\int_{S^2}{q_\i}
\hat{\gamma}_\i b d\omega M_*d\xi_*\right)1_{\{|A_\i|\le
N\}}1_{\left\{\left|\hat{g_\i^b}\right|\le L\right\}}
\\&\quad-e\Dv_\xi\left(\i(\i
E_\i+\xi\times B_\i)\left(g^b_\i\wedge L\right)\right)\\
&\quad+e\hat{\gamma}_\i\i E_\i\cdot\xi G_\i
\end{split}
\end{equation*}

Notice that $\left|\hat{\gamma}_\i G_\i\right|\le \f{27}{2}$, and
hence, by the interpolation between $L^1$ and $L^\infty$, we have
$$\left(\int_{\R^3}\int_{S^2}{q_\i} \hat{\gamma}_\i
b d\omega M_*d\xi_*\right)1_{\{|A_\i|\le
N\}}1_{\left\{\left|\hat{g_\i^b}\right|\le
L\right\}}+e\hat{\gamma}_\i\i E_\i\cdot\xi G_\i\in L^2(dtdxMdx)$$
and $$\Dv_\xi\left(\i(\i E_\i+\xi\times B_\i)\left(g^b_\i\wedge
L\right)\right)\in L^2(dtdx; H^{-1}(d\xi)).$$

Thus, from \eqref{11251}, we obtain
\begin{equation}\label{11254}
(\i\partial_t+\xi\cdot\nabla_x)\hat{g_\i^b}=\hat{\mathcal{S}}_\i\in
L^2(dtdxMdx)+L^2(dtdx; H^{-1}(d\xi).
\end{equation}
Applying the averaging theorem in \cite{DL1, GLPS}, we deduce from
\eqref{11254} that, for all $\chi(\xi)$ such that
$\f{\chi(\xi)}{1+|\xi|^2}\rightarrow 0$ as $|\xi|\rightarrow
\infty$,
\begin{equation}\label{11255}
\left\|\left<\hat{g_\i^b}\chi\right>\right\|_{L^2(0,T;
H^{\f{1}{4}}(\T))}\le C_{N,L},
\end{equation}
where $C_{N, L}$ depends only on $N, L$. This yields the
compactness of $<\hat{g_\i^b}\chi>$ in space; namely, there exists
a function $\eta: \R_+\mapsto \R_+$ such that $\lim_{z\rightarrow
0^+}\eta(z)=0$
\begin{equation}\label{11256}
\left\|\left<\hat{g_\i^b}\chi\right>(t, \cdot
+y)-\left<\hat{g_\i^b}\chi\right>(t,\cdot)\right\|_{L^2((0,T)\times\T)}\le\eta(|y|).
\end{equation}

{\bf{Step 3: Compactness of $g_\i^b\wedge L$.}} From \eqref{11254}
and the weak compactness of $q_\i$ in $L^1(dtdxd\mathcal{M})$, we
have, for large enough $N$,
$\left\|\ov{g_\i^b}(t,x,\xi)\right\|_{L^1(dtdxMd\xi)}$ can be as
small as we like. Thus, by \eqref{11256},
there exists a function $\eta: \R_+\mapsto \R_+$ with
$\lim_{z\rightarrow 0^+}\eta(z)=0$,
\begin{equation}\label{112561}
\left\|\left<\left(g_\i^b\wedge L\right)\chi\right>(t, \cdot
+y)-\left<\left(g_\i^b\wedge
L\right)\chi\right>(t,\cdot)\right\|_{L^1((0,T)\times\T)}\le\eta(|y|).
\end{equation}
Using the hypothesis that
$\left\{\left(g_\i^b\right)^2\right\}_{\{\i>0\}}$ is relatively
compact in $w-L^1(dt(1+|\xi|^2)Md\xi dx)$, we deduce easily that
there exists a function $\eta: \R_+\mapsto \R_+$ with
$\lim_{z\rightarrow 0^+}\eta(z)=0$,
\begin{equation}\label{1125611}
\left\|\left<\left(g_\i^b\wedge L\right)\chi\right>(t, \cdot
+y)-\left<\left(g_\i^b\wedge
L\right)\chi\right>(t,\cdot)\right\|_{L^2((0,T)\times\T)}\le\eta(|y|).
\end{equation}

{\bf{Step 4: Compactness of $g_\i^b$.}} Due to the hypothesis that
$\left\{\left(g_\i^b\right)^2\right\}_{\{\i>0\}}$ is relatively
compact in $w-L^1(dt(1+|\xi|^2)Md\xi dx)$, for every $\beta>0$,
there exists an integer $L>1$ such that
\begin{equation*}
\left\|\left<\left(g_\i^b\wedge L\right)\chi\right>(t, \cdot
)-\left<\left(g_\i^b
\right)\chi\right>(t,\cdot)\right\|_{L^2((0,T)\times\T)}\le
C\beta,
\end{equation*}
uniformly in $\i$. Thus, for such $\beta$ and $L$, we have
\begin{equation*}
\left\|\left<\left(g_\i^b\wedge L\right)\chi\right>(t, \cdot
+y)-\left<g_\i^b\chi\right>(t,\cdot+y)\right\|_{L^2((0,T)\times\T)}\le
C\beta,
\end{equation*}
and
\begin{equation*}
\left\|\left<\left(g_\i^b\wedge L\right)\chi\right>(t, \cdot
)-\left<g_\i^b\chi\right>(t,\cdot)\right\|_{L^2((0,T)\times\T)}\le
C\beta,
\end{equation*}
uniformly in $\i$. Hence, the above two inequalities, combining
together with \eqref{1125611}, imply there exists a function
$\eta: \R_+\mapsto \R_+$ such that $\lim_{z\rightarrow
0^+}\eta(z)=0$ and
\begin{equation*}
\int_0^T\int_{\T}\left|\left<g_{\i_n}^b\chi\right>(t,x+y)-\left<g_{\i_n}^b\chi\right>(t,x)\right|^2dxdt\le\eta(|y|)
\end{equation*}
for each $y\in\R^3$ such that $|y|\le 1$, uniformly in $n$.
\end{proof}

Now, we are ready to prove the convergence of the term of Lorentz
force.

\begin{Lemma}\label{lfc}
The following convergence holds in the sense of
distributions on $\R_+\times \R^3$:
$$B_\i\cdot\left<\xi\times\nabla_\xi\xi_k
g_\i^b\right>\rightarrow (B\times(\nabla\times B))_k,$$ as
$\i\rightarrow 0$, for all $1\le k\le 3$. The notation $a_k$
stands for the i-th component of the vector $a$. Further, we have
$j=e\u$.
\end{Lemma}
\begin{proof}
For any $1\le k\le 3$, $\nabla_\xi \xi_k=e_k$, where
$\{e_k\}_{k=1}^3$ is the standard basis for $\R^3$. This implies
$$\left<\xi\times\nabla_\xi\xi_k
g_\i^b\right>=\left<\xi g_\i^b\right>\times e_k.$$ Then, we can
rewrite $ B_\i\cdot<\xi\times\nabla_\xi\xi_k g_\i^b>$ as
\begin{equation}\label{ll31}
\begin{split}
B_\i\cdot\left<\xi\times\nabla_\xi\xi_k
g_\i^b\right>&=B_\i\cdot\left(\left<\xi g_\i^b\right>\times
e_k\right).
\end{split}
\end{equation}

Defining $$j_\i^b=e\f{\left<\xi(1+\i
g_\i^b)\right>}{\i}=e\left<\xi g_\i^b\right>,$$ since
$\left<\xi\right>=0$. Then, we have
\begin{equation}\label{ll310}
\left\|j_\i^b-\f{j_\i}{\i}\right\|_{L^\infty_t(L^1(dxMd\xi))}\rightarrow 0,
\end{equation}
as $\i\rightarrow 0$. Indeed, from the definition of $g_\i^c$, we
know that $\i g_\i^c$ is uniformly bounded in
$L^\infty_t(L^2(dxMd\xi))$ while from the second statement of
Lemma \ref{ec}, $g_\i^c$ is uniformly bounded in
$L^\infty_t(L^1(dxMd\xi))$. Thus, by the interpolation between
$L^2$ and $L^1$, we deduce that
$$\|\i g_\i^c\|_{L^\infty_t(L^{\f{3}{2}}(dx Md\xi))}\le
C\i^{\f{1}{2}},$$ for some constant $C>0$. Therefore, we have
\begin{equation*}
\begin{split}
\left\|j_\i^b-\f{j_\i}{\i}\right\|_{L^\infty_t(L^1(dxMd\xi))}&=\left\|g_\i^b\xi-g_\i \xi\right\|_{L^\infty_t(L^1(dxMd\xi))}\\
&=\left\|\i g_\i^c\xi\right\|_{L^\infty_t(L^1(dxMd\xi))}\\
&\le \left\|\i
g_\i^c\right\|_{L^\infty_t(L^{\f{3}{2}}(dxMd\xi))}<|\xi|^3>^{\f{1}{3}}\\
&\le C\i^{\f{1}{2}}\rightarrow 0,
\end{split}
\end{equation*}
as $\i\rightarrow 0$. Hence, \eqref{ll310}, combining with the
weak convergence of $\left\{\f{j_i}{\i}\right\}$ in
$L^\infty_t(L^2(dxMd\xi))$ and the uniform bound of
$\{j_\i^b\}$ in $L^\infty_t(L^2(dxMd\xi))$, implies
that $j_\i^b$ converges weakly to $j$ in
$L^\infty_t(L^2(dxMd\xi))$. Note that
$\f{j_\i}{\i}=e\left<g_\i\xi\right>$, we have $j=e\u$.

Notice that, \eqref{me33} implies
$$\partial_t B_\i=-\nabla\times
E_\i\in L^\infty(0,T; W^{-4,2}(\T))\subset L^1(0,T;
W^{-s,1}(\T)),$$ for some $s>4$ large enough, and is bounded in
$L^1(0,T; W^{-s,1}(\T))$ uniformly in $\i$. On the other hand,
Lemma \ref{ll12} with $\chi(\xi)=\xi_k$ implies that for each
$T>0$,
\begin{equation}\label{lf1}
\int_0^{T}\int_\T|\left<\xi g_\i^b\right>(t,x+y)\times
e_k-\left<\xi g_\i^b\right>(t,x)\times e_k|^2dxdt\le \eta(|y|),
\end{equation}
for each $y\in \R^3$ such that $|y|\le 1$, uniformly in $\i$,
where $\eta$ is a function $\R_+\mapsto \R_+$ satisfying
$\lim_{z\rightarrow 0^+}\eta(z)=0.$ Hence, by Lemma 5.1 in
\cite{PLL}, one has
\begin{equation}\label{ll32}
(B_\i)\cdot\left(\left<\xi g_\i^b\right>\times
e_k\right)\rightarrow B\cdot (e\u\times e_k)=B\cdot (j\times
e_k)=(B\times j)_k,
\end{equation}
in the sense of distributions.
The proof is complete.
\end{proof}

\bigskip\bigskip

\section*{Acknowledgments}

D. Wang's research was supported in part by the National Science
Foundation under Grant DMS-0906160 and by the Office of Naval
Research under Grant N00014-07-1-0668.

\end{document}